\RequirePackage{fix-cm}

\documentclass{svjour3}                     

\smartqed  

\usepackage{graphicx}
\usepackage{amsmath}
\usepackage{amssymb}
\usepackage{algorithm}
\usepackage{algpseudocode}
\usepackage{color}
\usepackage{bbm}

\usepackage{stackengine}
\def\delequal{\mathrel{\ensurestackMath{\stackon[1pt]{=}{\scriptstyle\Delta}}}}
\DeclareMathOperator*{\argmin}{arg\,min}

\graphicspath{{figures/}}

\begin{document}

\title{SPRSF: Sparse Phase Retrieval via Smoothing Function}

\author{Samuel Pinilla\and
        Jorge Bacca \and Henry Arguello 
}

\institute{Samuel Pinilla \at
              \email{samuel.pinilla@correo.uis.edu.co}           
}

\date{ }

\maketitle

\begin{abstract}
Phase retrieval (PR) is an ill-conditioned inverse problem which can be found in various science and engineering applications. Assuming sparse priority over the signal of interest, recent algorithms have been developed to solve the phase retrieval problem. Some examples include SparseAltMinPhase (SAMP), Sparse Wirtinger flow (SWF) and Sparse Truncated Amplitude flow (SPARTA). However, the optimization cost functions of the mentioned algorithms are non-convex and non-smooth. In order to fix the non-smoothness of the cost function, the SPARTA method uses truncation thresholds to calculate a truncated step update direction. In practice, the truncation procedure requires calculating more parameters to obtain a desired performance in the phase recovery. Therefore, this paper proposes an algorithm called SPRSF (Sparse Phase retrieval via Smoothing Function) to solve the sparse PR problem by introducing a smoothing function. SPRSF is an iterative algorithm where the update step is obtained by a hard thresholding over a gradient descent direction. Theoretical analyses show that the smoothing function uniformly approximates the non-convex and non-smooth sparse PR optimization problem. Moreover, SPRSF does not require the truncation procedure used in SPARTA. Numerical tests demonstrate that SPRSF performs better than state-of-the-art methods, especially when there is no knowledge about the sparsity $k$. In particular, SPRSF attains a higher mean recovery rate in comparison with SPARTA, SAMP and SWF methods, when the sparsity varies for the real and complex cases. Further, in terms of the sampling complexity, the SPRSF method outperforms its competitive alternatives.
\end{abstract}

\section{Introduction}
\label{intro}
In many applications of science and engineering, it is required to recover a signal from the squared modulus of any linear transform, which is known as phase retrieval (PR). Such a task is present in optics \cite{xu2015overcoming}, astronomical imaging \cite{fienup1987phase}, microscopy \cite{mayo2003x} and x-ray crystallography \cite{millane1990phase,pinilla2018coded1,pinilla2018coded}, where  the optical sensors measure the intensities of the reflection, but they are not able to measure the phase of the signal. For example, in x-ray crystallography \cite{millane1990phase}, PR is used to determine the atomic position of a crystal in a three-dimensional (3D) space \cite{xcrys}. Recent approaches propose different measurement transforming systems such as over-sampling Fourier, short-time Fourier, random Gaussian, and coded diffraction systems, which have drawn attention, since they combine some active fields such as x-ray imaging, coded diffractive imaging and phase retrieval techniques \cite{chen2015solving,candes2015phase}.

Previous algorithms for PR are based on the error-reduction method \cite{fienup1982phase} which was proposed in 1970. However, these algorithms do not have theoretical guarantees and their rates of convergence are considerably slow \cite{candes2015phase,fienup1982phase}. Most recent approaches can be grouped as convex and non-convex. In particular, a convex formulation was proposed in \cite{candes2014solving} via Phaselift, which consists on lifting up the original vector recovery problem from a quadratic system into a rank-1 matrix recovery. Further, theoretical guarantees of convergence and recovery  for the convex approach have been developed, but its computational complexity becomes extremely high for large signals. On the other hand, one of the non-convex formulations, called Wirtinger Flow (WF), is a gradient descent method based on the Wirtinger derivative, which was demonstrated to allow exact recovery from the phaseless measurements up to a global unimodular constant\cite{candWir}. Also, alternating non-convex projection algorithms have been proposed such as AltMinPhase \cite{netrapalli2013phase}, Truncated Amplitude Flow TAF \cite{wang2016solving}, the Wirtinger Flow (WF) variants \cite{candWir,cande2}, as well as trust-region-methods \cite{sun2016geometric}.

In \cite{jaganathan2015phase} it was shown that some signals in PR are naturally sparse. Further, enforcing sparsity constraints can also ensure uniqueness of the discretized one-dimensional PR \cite{fienup1982phase}. Moreover, the AltMinPhase, WF and TAF recovery methods have been extended to PR of sparse inputs yielding solvers such as SparseAltMinPhase (SAMP) \cite{netrapalli2013phase}, sparse WF (SWF) \cite{yuan2017phase}, and the Sparse Truncated Amplitude flow (SPARTA) \cite{wang2016sparse}, respectively. The SAMP, SWF and SPARTA methods use different initializations strategies in order to guarantee exact recovery of the true signal. In particular, SAMP uses the spectral initialization strategy introduced in \cite{candWir}, SPARTA introduces the sparse orthogonally promoting initialization in \cite{wang2016sparse} and SWF proposes a variant of the spectral initialization developed in \cite{yuan2017phase}. The initialization for the SWF algorithm returns a more accurate estimation of the true signal, in comparison with the SPARTA and the SAMP initializations. An important characteristic of these sparse PR solvers, is that the optimization functions are non-convex and, in the case of SPARTA and SAMP, they are also non-smooth. Further, in order to fix the non-smoothness of the cost function, SPARTA uses truncation thresholds to calculate a truncated step update direction. But, in practice this requires to calculate more parameters to obtain a desired performance in recovering the phase. Further, the truncation procedure drastically modifies the search direction update, which increases the sampling complexity to recover the phase. In summary, the SPARTA, SWF and SAMP algorithms optimize non-convex cost functions and also non-smooth in the case of SAMP and SPARTA. Moreover, SPARTA requires an extra truncation procedure in the gradient step, which requires the design of more parameters to obtain a desired performance to recover the phase.

On the other hand, \cite{zhang2009smoothing} introduced the concept of a smoothing function for a non-smooth and non-convex optimization problem on a closed convex set. A smoothing function is a smooth approximation of the original non-convex and non-smooth optimization cost function \cite{zhang2009smoothing}, which in this case is the sparse PR problem. Therefore, given that the sparse PR can be formulated as a non-convex and non-smooth optimization problem, this paper proposes an algorithm called Sparse Phase Retrieval via Smoothing Function (SPRSF) to solve the sparse PR problem by introducing a smoothing function. SPRSF is an iterative algorithm where the update step is obtained by a hard thresholding over a gradient descent direction. Theoretical analyses show that the smoothing function uniformly approximates the non-convex and non-smooth sparse PR optimization problem. Moreover, it is proved that SPRSF converges linearly for any $k$-sparse $n$-long signal $(k\ll n)$ with sampling complexity $\mathcal{O}(k^{2}\log(n))$. Moreover, SPRSF does not require the truncation procedure used in SPARTA. Numerical tests demonstrate that SPRSF performs better than state-of-the-art methods specially when there is no knowledge about the sparsity $k$. Further, it is shown that the SPRSF method outperforms its competitive alternatives SAMP, SPARTA and SWF algorithms in terms of sampling complexity.

\section{Sparse Phase Retrieval Problem}
The sparse phase retrieval problem can be formulated as the solution to the system of $m$ quadratic equations of the form
\begin{equation}
y_{i} = \lvert \langle \mathbf{a}_{i},\mathbf{x} \rangle \rvert ^2, i=1,\cdots,m, \mbox{ subject to } \lVert \mathbf{x} \rVert_{0} = k,
\label{eq:basicProblem}
\end{equation}
with data vector $\mathbf{y}:=[y_{1},\cdots,y_{m}]^{T}\in \mathbb{R}^{m}$, $\mathbf{a}_{i}\in \mathbb{R}^{n}/\mathbb{C}^{n} $ are the known sampling vectors, $\mathbf{x}\in \mathbb{R}^{n}/\mathbb{C}^{n}$ is the desired unknown signal, the sparsity level $k\ll n$ is assumed to be known and $\lVert \cdot\rVert_{0}$ is the zero pseudo-norm. This work considers the complex-valued Gaussian vectors $\mathbf{a}_{i}\sim \mathcal{CN}(0,\mathbf{I}_{n})=\mathcal{N}(0,\frac{1}{2}\mathbf{I}_{n})+j\mathcal{N}(0,\frac{1}{2}\mathbf{I}_{n})$, assumed to be independently and identically distributed (i.i.d.), where $j=\sqrt{-1}$. Then, adopting the least-squares criterion, the task of recovering a $k$-sparse solution from phaseless equations, as in \eqref{eq:basicProblem}, reduces to that of minimizing the amplitude-based loss function
\begin{equation}
\min_{\lVert \mathbf{x} \rVert_{0}=k} f(\mathbf{x})=\frac{1}{m}\sum_{i=1}^{m} \left(f_{i}(\mathbf{x})-q_i  \right)^{2},
\label{eq:traditional}
\end{equation}
where $f_{i}(\mathbf{x}) = \lvert \langle \mathbf{a}_{i},\mathbf{x} \rangle \rvert$ and $q_i = \sqrt{y_{i}}$. However, notice that the optimization problem in \eqref{eq:traditional} is  non-smooth and non-convex \cite{candes2014solving}. Thus, this work proposes an algorithm which introduces an auxiliary smooth function $g(\cdot)$ to approximate the original objective function $f(\cdot)$, in order to solve the non-smooth and non-convex optimization problem in \eqref{eq:traditional}. For this, some conditions over the auxiliary function $g(\cdot)$ are required, but these will be discussed in Section \ref{sec:algorithm}.

Throughout the paper the following notations are considered. The set $\mathbb{R}_{+}=\{x\in \mathbb{R}: x\geq 0\}$ and the set $\mathbb{R}_{++}=\{x\in \mathbb{R}: x>0\}$. We denote $\mathbf{w}^{H}\in \mathbb{C}^{n}$ as the conjugate transpose version of the vector $\mathbf{w}\in \mathbb{C}^{n}$, and the distance between any two complex vectors $\mathbf{w}_{1},\mathbf{w}_{2}\in \mathbb{C}^{n}$ as 
\begin{equation}
d_{r}(\mathbf{w}_{1},\mathbf{w}_{2})=\min_{\theta \in [0,2\pi)} \lVert \mathbf{w}_{1}e^{-j\theta}-\mathbf{w}_{2}\rVert_{2},
\label{eq:distance}
\end{equation}
where $\lVert \cdot \rVert_{2}$ denotes the Euclidean norm. Note that the distance $d_{r}(\cdot,\cdot)$ defined in \eqref{eq:distance} reduces to computing $d_{r}(\mathbf{w}_{1},\mathbf{w}_{2}):=\min\lVert \mathbf{w}_{1}\pm \mathbf{w}_{2} \rVert_{2}$ for $\mathbf{w}_{1},\mathbf{w}_{2}\in \mathbb{R}^{n}$.

\section{Sparse Phase Retrieval algorithm}
\label{sec:algorithm}
The concept of the smoothing function was presented in \cite{zhang2009smoothing} as Definition \ref{def:smoothDef}, which is an important notion to the proposed algorithm. First, the concept of a locally Lipschitz continuous function is presented.

\begin{definition}{\textit{Lipschitz continuous under the distance $d_{r}(\cdot,\cdot)$}:}
	Let $f:\left(\mathbb{C}^{n},d_{r}(\cdot,\cdot)\right) \rightarrow \mathbb{R}$ be a function. The function $f$ is called Lipschitz continuous if there exists a constant $L>0$ such that, for all $\mathbf{w}_{1}, \mathbf{w}_{2}\in \mathbb{C}^{n}$
	\begin{equation}
	\lvert f(\mathbf{w}_{1})-f(\mathbf{w}_{2}) \rvert \leq L \text{ } d_{r}( \mathbf{w}_{1},\mathbf{w}_{2}).
	\end{equation}
	\label{def:lipschitz}
	\vspace{-1.2em}
\end{definition} 

\begin{definition}{\textit{Locally Lipschitz continuous under the distance $d_{r}(\cdot,\cdot)$}:}
	Let $f:\left(\mathbb{C}^{n},d_{r}(\cdot,\cdot)\right) \rightarrow \mathbb{R}$ be a function. The function $f(\cdot)$ is called Locally Lipschitz continuous if for every $\mathbf{w}\in \mathbb{C}^{n}$ exists a neighborhood $\mathcal{U}$, such that, $f(\cdot)$ restricted to $\mathcal{U}$ is Lipschitz continuous.
	\label{def:lipschitz1}
\end{definition}

\begin{definition}{\textit{Smoothing function:}}
	Let $f: \mathbb{C}^{n}\rightarrow \mathbb{R}$ be a locally Lipschitz continuous function. Then $g:\mathbb{C}^{n}\times \mathbb{R}_{+}\rightarrow \mathbb{R}$ is a smoothing function of $f(\cdot)$ if $g(\cdot,\mu)$ is smooth in $\mathbb{C}^{n}$ for any given $\mu\in \mathbb{R}_{++}$ and 
	\begin{equation}
	\lim_{\mu \downarrow 0} g(\mathbf{w},\mu)=f(\mathbf{w}),
	\end{equation}
	for any fixed $\mathbf{w}\in \mathbb{C}^{n}$.
	\label{def:smoothDef}
\end{definition}

According to Definition \ref{def:smoothDef}, consider the function $\varphi_{\mu}:\mathbb{R} \rightarrow \mathbb{R}_{++}$ defined as
\begin{equation}
\varphi_{\mu}(x) = \sqrt{x^{2}+\mu^{2}},
\label{eq:realApprox}
\end{equation}
where $\mu\in \mathbb{R}_{++}$. Notice that $\varphi_{\mu}(\cdot)$ approximates the function $f_{i}(\cdot)$ in \eqref{eq:traditional}, because $\varphi_{0}(\lvert \mathbf{a}_{i}^{H}\mathbf{x}\rvert)=f_{i}(\mathbf{x})$. Also, according to Definition \ref{def:lipschitz1}, it is necessary to prove that the objective function $f(\mathbf{x})$ in \eqref{eq:traditional} is locally Lipschitz continuous. Thus, Lemma \ref{lem:locallyLipschi} shows that $f(\mathbf{x})$ is locally Lipschitz.
\begin{lemma}
	The function $f(\mathbf{x})$ in \eqref{eq:traditional} is locally Lipschitz continuous under the distance $d_{r}(\cdot,\cdot)$ with probability at least $1-me^{-n/2}$.
	\label{lem:locallyLipschi}
\end{lemma}
\begin{proof}
	The proof of this lemma can be found in Appendix A.
\end{proof}

The following Lemma \ref{lem:uniformlyApprox} shows that $\varphi_{\mu}(\cdot)$ has important smooth properties to approximate the functions $f_{i}(\cdot)$, given that $\varphi_{0}(\lvert\mathbf{a}_{i}^{H}\mathbf{x}\rvert)=f_{i}(\mathbf{x})$.

\begin{lemma}
	The function $\varphi_{\mu}(x)$, defined in \eqref{eq:realApprox}, converges uniformly to $\varphi_{0}(x)$ on $\mathbb{R}$.
	\label{lem:uniformlyApprox}
\end{lemma}
\begin{proof}
	According to the definition of the function $\varphi_{\mu}$ in \eqref{eq:realApprox}, it can be obtained that
	\begin{equation}
	\lvert \varphi_{\mu}(x)-\varphi_{0}(x) \rvert = \lvert \sqrt{x^{2}+\mu^{2}}-\sqrt{x^{2}} \rvert.
	\end{equation}
	Notice that by the Minkowski inequality \cite{kreyszig1989introductory}, it can be expressed that $\sqrt{x^{2}+\mu^{2}}\leq \sqrt{x^{2}}+\mu$, therefore 
	\begin{equation}
	\lvert \varphi_{\mu}(x)-\varphi_{0}(x) \rvert \leq \lvert \sqrt{x^{2}}+\mu-\sqrt{x^{2}} \rvert\leq \mu.
	\label{eq:uniformConver}
	\end{equation}
\end{proof}

Lemma \ref{lem:uniformlyApprox} establishes that the function $\varphi_{\mu}(\lvert\mathbf{a}_{i}^{H}\mathbf{x}\rvert)$ uniformly approximates $\varphi_{0}(\lvert\mathbf{a}_{i}^{H}\mathbf{x}\rvert)=f_{i}(\mathbf{x})$ in \eqref{eq:traditional}, which is a desirable convergence, since it only depends on the value of $\mu$. Thus, a smooth optimization problem to recover the unknown desired signal $\mathbf{x}\in \mathbb{C}^{n}/\mathbb{R}^{n}$ from the measurements $q_{i}$ in \eqref{eq:traditional} can be formulated as
\begin{equation}
\min_{\lVert \mathbf{x} \rVert_{0}=k} g(\mathbf{x},\mu)=\frac{1}{m}\sum_{i=1}^{m} \left(\varphi_{\mu}(\lvert\mathbf{a}_{i}^{H}\mathbf{x}\rvert)-q_{i}\right)^{2},
\label{eq:optimizationProblem}
\end{equation}
where $g(\mathbf{x},\mu)$ is the smoothing function of $f(\mathbf{x})$ in \eqref{eq:traditional}.

To solve \eqref{eq:optimizationProblem}, this work proposes the Sparse Phase Retrieval algorithm via Smoothing Function (SPRSF), summarized in Algorithm \ref{alg:smothing}. SPRSF is a gradient thresholded descent method, which iteratively refines a initial guess solution. Specifically, in Line 2 the algorithm calculates the initial guess $\mathbf{z}^{(0)}$, procedure that will be explained in Subsection \ref{sub:initialGuess}. Also,  following the algorithm in each iteration, the thresholded step is calculated in Line 4 as will be explained in Subsection \ref{sub:threshold}. Further, the smoothing parameter is updated to obtain a new point. That is, if  $\left\lVert \partial g\left(\mathbf{z}^{(t+1)},\mu_{(t)} \right) \right\rVert_{2} \geq \gamma \mu_{(t)}$, in Line 5 is not satisfied, then the smoothing parameter is updated using the new point in Line 8.  Each vector $\partial g(\mathbf{z}^{(t)},\mu_{(t)})$ in Algorithm \ref{alg:smothing} is calculated using the Wirtinger derivative as was introduced in \cite{hunger2007introduction}. The following definition establishes the Wirtinger derivative of the function $g(\mathbf{x},\mu)$.

\begin{algorithm}[h]
	\caption{Sparse Phase Retrieval Algorithm via Smoothing Function (SPRSF)}
	\label{alg:smothing}
	\begin{algorithmic}[1]
		\State {\textbf{Input: } Data $\{(\mathbf{a}_{i};q_{i})\}_{i=1}^{m}$, sparsity level $k$. The step size $\tau\in (0,1)$, control variables $\gamma,\gamma_{1}\in (0,1)$, $\mu_{(0)}\in \mathbb{R}_{++}$ and number of iterations $T$.}
		\Statex{}
		\State{\textbf{Initialization: }  $S_{0}$ set to be the set of $k$ largest indices of $\{\frac{1}{m}\sum_{i=1}^{m}q_{i}^{2}a^{2}_{i,j}\}_{1\leq j\leq n}$. Let $\tilde{\mathbf{x}}^{(0)}$ be the leading eigenvector of the matrix $\mathbf{Y} := \frac{1}{ m} \sum_{i \in  I_{0}} \sqrt{q_{i}}\frac{\mathbf{a}_{i,S_{0}}\mathbf{a}_{i,S_{0}}^{H}}{\lVert\mathbf{a}_{i,S_{0}}\rVert_{ 2}^{2}}$}. Define the initial point as $\mathbf{z}^{(0)} := \lambda_{0}\tilde{\mathbf{x}}^{(0)}$, where $\lambda_{0} := \sqrt{\frac{\sum_{i=1}^m q^{2}_i}{m}}$.
		\Statex{}
		\For{$t=0:T-1$} 
		\State{$\mathbf{z}^{(t+1)} = \mathcal{H}_{k}(\mathbf{z}^{(t)} - \tau\partial g(\mathbf{z}^{(t)},\mu_{(t)}) )$ \Comment{\footnotesize{Thresholded step}}}
		\If {$\lVert \partial g\left(\mathbf{z}^{(t+1)},\mu_{(t)} \right) \rVert_{2} \geq \gamma \mu_{(t)} $}
		\State $\mu_{(t+1)}=\mu_{(t)}$
		\Else
		\State $\mu_{(t+1)}= \gamma_1 \mu_{(t)}$ \Comment{\footnotesize{Smoothing parameter actualization}}
		\EndIf
		\EndFor
		\State{\textbf{Output: } $\mathbf{z}^{(T)}$}
	\end{algorithmic}
\end{algorithm}

\begin{definition}{\textit{Wirtinger derivative }\cite{hunger2007introduction}}:
	The Wirtinger derivative of a real-valued function $h(\mathbf{w}):\mathbb{C}^{n}\rightarrow \mathbb{R}$ with complex-valued argument $\mathbf{w}\in \mathbb{C}^{n}$ can be computed as
	\begin{align}
	\partial h(\mathbf{w}) \delequal 2\frac{\partial h(\mathbf{w})}{\partial \mathbf{w}^{*}} = 2\left[ \frac{\partial h(\mathbf{w})}{\partial w^{*}_{1}},\cdots, \frac{\partial h(\mathbf{w})}{\partial w^{*}_{n}}\right]^{T}
	\end{align}
	where $w^{*}_{i}$ denotes the conjugate of $w_{i}$. More details related to Wirtinger derivation can be found in \cite{hunger2007introduction}. Note that this derivation has been recently used in state-of-the-art methods to solve the phase retrieval problem \cite{candWir,wang2016solving,cande2}.
	\label{def:Wirtinger}
\end{definition}

SPRSF applies gradient iterations based on the Wirtinger derivative, introduced in Definition \ref{def:Wirtinger}, to refine the initial estimate. Specifically, the Wirtinger derivative of $g(\mathbf{z}^{(t)},\mu_{(t)})$ is given by
\begin{equation}
\partial g(\mathbf{z}^{(t)},\mu_{(t)}) = \frac{2}{m}\sum_{i=1}^{m}\left(\mathbf{a}_{i}^{H}\mathbf{z}^{(t)}-q_{i} \frac{\mathbf{a}_{i}^{H}\mathbf{z}^{(t)}}{\sqrt{\lvert\mathbf{a}_{i}^{H}\mathbf{z}^{(t)}\rvert^{2}+\mu_{(t)}^{2}}} \right)\mathbf{a}_{i}.
\label{eq:gradient}
\end{equation}

Notice that, in contrast to the gradient update steps for the SPARTA method introduced in \cite{wang2016sparse}, $\partial g(\mathbf{z}^{(t)},\mu_{(t)})$ in \eqref{eq:gradient} is always continuous because $\mu_{(t)}\not =0$ for any $t\in \mathbb{N}$. Therefore, the proposed SPRSF method does not require any truncation parameter.

\subsection{Initialization Stage}
\label{sub:initialGuess}
The initialization in PR is a crucial step in order to increase the speed of convergence and reduce the sampling complexity \cite{wang2016sparse,zhang2016reshaped}. Then, this work uses the Weighted Maximal Correlation initialization proposed in \cite{wang2017solving}. This initialization consists in calculating the vector $\mathbf{z}^{(0)}$ as the leading eigenvector $\tilde{\mathbf{x}}^{(0)}$ of the matrix
\begin{equation}
\mathbf{Y} := \frac{1}{ m} \sum_{i \in  I_{0}} \sqrt{q_{i}}\frac{\mathbf{a}_{i,S_{0}}\mathbf{a}_{i,S_{0}}^{H}}{\lVert\mathbf{a}_{i,S_{0}}\rVert_{ 2}^{2}},
\label{eq:matrixIni}
\end{equation}
scaled by the quantity $\lambda_{0} := \sqrt{\frac{\sum_{i=1}^m q^{2}_i}{m}}$, $i.e,$ $\mathbf{z}^{(0)} = \lambda_{0} \tilde{\mathbf{x}}^{(0)}$. The set $S_{0}$ is the estimated support of the signal $\mathbf{x}$ which is calculated using the same approach introduced in \cite{wang2016sparse}. Specifically, $S_{0}$ is the set of the $k$ largest indices of $\{\frac{1}{m}\sum_{i=1}^{m}q_{i}^{2}a_{i,j}\}_{1\leq j \leq n}$. The set $I_{0}$ is the collection of indices corresponding to the largest values of $\{ \lvert\langle \mathbf{a}_{i},\mathbf{x} \rangle\rvert / \Vert \mathbf{a}_{i}\Vert_{2} \}$. The notation $\lvert I_{0} \rvert$ is the cardinality of the set $I_{0}$ which is usually chosen as $\lfloor\frac{3m}{13}\rfloor$, where $\lfloor w \rfloor$ denotes the largest integer number smaller than $w$.

Moreover, in \cite{wang2016sparse} it was established that the distance between the initial guess $\mathbf{z}^{(0)}$ and the true signal $\mathbf{x}$ is given by
\begin{equation}
d_{r}(\mathbf{z}^{(0)},\mathbf{x})\leq \delta\lVert \mathbf{x}\rVert_{2},
\label{eq:distanceInit}
\end{equation}
with probability not less than $1-\exp(-C_{0}m)$, providing that $m\geq c_{0}k^{2}\log(mn)$ for some constant $C_{0}$ and $c_{0}>0$ which is determined by $\delta \in (0,1)$. The initialization procedure is calculated in Line 2 of Algorithm \ref{alg:smothing}. 

\subsection{Thresholded Gradient Stage}
\label{sub:threshold}
The proposed Algorithm \ref{alg:smothing} solves the sparsity constraint of the optimization problem in \eqref{eq:optimizationProblem} by iteratively refining the current update step $\mathbf{z}^{(t)}$ by a $k$-sparse hard thresholding operator $\mathcal{H}_{k}(\cdot)$, as calculated in Line 4 in Algorithm \ref{alg:smothing}. Specifically, $\mathcal{H}_{k}(\mathbf{u})$ sets all the entries in the vector $\mathbf{u}\in \mathbb{C}^{n}$ to zero, except for its $k$ largest absolute values.

\subsection{Convergence Conditions}
This subsection provides theoretical results that guarantee the convergence of the proposed method summarized in Algorithm \ref{alg:smothing}. The following theorem establishes that the successive estimates of SPRSF in Line 4 of Algorithm \ref{alg:smothing}, tend to the unknown desired signal $\mathbf{x}\in \mathbb{C}^{n}$ for a given value of $\mu$.

\begin{theorem}{(\textit{Local error contraction}): }
	Let $\mathbf{x}\in \mathbb{C}^{n}$ be any $k$-sparse ($k\ll n$) signal vector with the minimum nonzero entry on $(1/\sqrt{k})\lVert \mathbf{x} \rVert_{2}$. Consider the measurements $q_{i}=\lvert \langle \mathbf{a}_{i},\mathbf{x} \rangle \rvert$, where $\mathbf{a}_{i}\sim \mathcal{CN}(0,\mathbf{I}_{n}), \forall i=1,\cdots,m$. With a constant
	step size $\tau \in (0,1)$, successive estimates of SPRSF in Algorithm \ref{alg:smothing} satisfy
	\begin{equation}
	d_{r}(\mathbf{z}^{(t+1)},\mathbf{x}) \leq \delta(1-\eta)^{t+1} \lVert \mathbf{x} \rVert_{2}
	\label{eq:linearConvergence}
	\end{equation}
	which holds with probability exceeding $1-2 e^{-c_{1}m}$ provided that $m\geq C_{1}k^{2}\log(mn)$. Here, $c_{1},C_{1}\geq 0$ and $0 < \eta <1$ are some universal constants. The constant $\delta$ is obtained from \eqref{eq:distanceInit}
	\label{theo:contrac}
\end{theorem}

\begin{proof}
	The proof of Theorem \ref{theo:contrac} can be found in Appendix B.
\end{proof}

Notice that Theorem \ref{theo:contrac} only provides that the sequence $\{\mathbf{z}^{(t)}\}_{t\geq 1}$, generated by Algorithm \ref{alg:smothing}, produces a monotonically decreasing sequence $\{g(\mathbf{z}^{(t)},\mu)\}_{t\geq 1}$, with a given $\mu$. Moreover, the sampling complexity bound $m\geq C_{1}k^{2}\log(mn)$, can often be rewritten as $m\geq C_{1}'k^{2}\log(n)$ for some constant $C_{1}'\geq C_{1}$ and large enough $n$ \cite{wang2016sparse}. Thus, it can be concluded that the sampling complexity of the SPRSF algorithm is $\mathcal{O}(k^{2}\log(n))$.

On the other hand, in order to prove that the proposed method solves the original optimization problem in \eqref{eq:basicProblem} it must be shown that $\{\mu_{(t)}\}_{t\geq 1}$ tends to zero $i.e.$ $\mu_{(t)}\rightarrow 0$. Thus, Theorem \ref{theo:assumption} establishes the sufficient conditions to guarantee that $\mu_{(t)}\rightarrow 0$, which are used in Theorem \ref{eq:convergence} to guarantee the convergence of Algorithm \ref{alg:smothing}. 
\begin{theorem}
	\label{theo:assumption}
	Assuming that $\text{span}(\mathbf{a}_{1},\cdots,\mathbf{a}_{m})=\{\sum_{k=1}^{m}\lambda_{k}\mathbf{a}_{k}: \lambda_{k}\in \mathbb{C} \}=\mathbb{C}^{n}$, then functions $\varphi_{\mu}$ and $g(\mathbf{x},\mu)$ defined in \eqref{eq:optimizationProblem} satisfy the following properties:
	\begin{enumerate}
		\item For any $\left(\mathbf{w},\mu\right)\in \mathbb{C}^{n}\times \mathbb{R}_{++}$, the level set
		\begin{equation}
		S_{\mu}(\mathbf{w}) = \{\mathbf{z}\in \mathbb{C}^{n} | g\left(\mathbf{z},\mu\right) \leq g\left(\mathbf{w},\mu\right) \},
		\label{eq:subLevel}
		\end{equation}
		is bounded.
		\item The Wirtinger derivative $\partial g(\mathbf{z},\mu)$ with respect to $\mathbf{z}$ is smooth and there exists a constant $L_{g}>0$, such that, for any $\mathbf{w}\in \mathbb{C}^{n}$ and a given $\mu\in \mathbb{R}_{++}$ it is satisfied that
		\begin{equation}
		d_{r}( \partial g(\mathbf{z}_{1},\mu),\partial g(\mathbf{z}_{2},\mu))\leq L_{g}d_{r}( \mathbf{z}_{1},\mathbf{z}_{2}),
		\label{eq:assumption}
		\end{equation}
		for all $\mathbf{z}_{1},\mathbf{z}_{2}\in S_{\mu}(\mathbf{w})$ with probability at least $1-me^{-n/2}$.
	\end{enumerate}
\end{theorem}
\begin{proof}
	The proof of Theorem \ref{theo:assumption} can be found in Appendix C.
\end{proof}

Finally, based on Theorem \ref{theo:assumption}, Theorem \ref{eq:convergence} establishes that the sequence $\mu_{(t)}$ tends to zero, which combined with Theorem \ref{theo:contrac} proves that Algorithm \ref{alg:smothing} solves the optimization problem in \eqref{eq:basicProblem}. 

\begin{theorem}
	Under the setup of Theorems \ref{theo:contrac} and \ref{theo:assumption}, the sequences $\{\mu_{(t)}\}$ and $\{\mathbf{z}^{(t)}\}$ in Algorithm \ref{alg:smothing} satisfy $\lim_{t\rightarrow \infty} \mu_{(t)}=0$ and $\liminf_{t\rightarrow \infty}\lVert \partial g(\mathbf{z}^{(t)},\mu_{(t-1)}) \rVert_{2}=0$.
	\label{eq:convergence}
\end{theorem}
\begin{proof}
 The proof of Theorem \ref{eq:convergence} is deferred to Appendix D.	
\end{proof}

\section{Advantages of the Proposed Approach}
This section is devoted to analyze why the smooth cost function in \eqref{eq:optimizationProblem} does not need a truncation procedure in its update rule. Notice that the decent direction (the Wirtinger derivative) in \eqref{eq:gradient} for each $t$-th iteration in Algorithm \ref{alg:smothing} can be rewritten as 
\begin{align}
\partial g(\mathbf{z}^{(t)},\mu_{(t)}) =  \frac{2}{m}\sum_{i=1}^{m}\left(1 - \frac{\lvert\mathbf{a}_{i}^{H}\mathbf{x}\rvert}{\sqrt{\lvert\mathbf{a}_{i}^{H}\mathbf{z}^{(t)}\rvert^{2}+\mu_{(t)}^{2}}}\right)\mathbf{a}_{i}\mathbf{a}_{i}^{H}\mathbf{z}^{(t)}.
\label{eq:nontrunc2}
\end{align} 
Then, considering the update procedure of the variable $\mu$ in Algorithm \ref{alg:smothing} we have that
\begin{align}
\lVert \partial g(\mathbf{z}^{(t)},\mu_{(t-1)}) \rVert_{2} \leq \gamma \mu_{(t)} \leq \mu_{(t)},
\label{eq:ineGradient}
\end{align}
for some $\gamma\in (0,1)$. Further, in Theorem \ref{eq:convergence} it is established that from \eqref{eq:ineGradient} the Wirtinger derivative in \eqref{eq:nontrunc2} tends to zero. Thus, from the result in Theorem \ref{eq:convergence} and inequality \eqref{eq:ineGradient}, it can be concluded that
\begin{align}
\left \lvert 1-\frac{\lvert\mathbf{a}_{i}^{H}\mathbf{x}\rvert}{\sqrt{\lvert\mathbf{a}_{i}^{H}\mathbf{z}^{(t)}\rvert^{2}+\mu_{(t)}^{2}}} \right \rvert<1,
\end{align}
for all $i\in \{1,\cdots,m\}$, because otherwise inequality \eqref{eq:ineGradient} does not hold (see Appendix D). For this reason, the Wirtinger gradient in \eqref{eq:nontrunc2}, used by the proposed method, does not need truncation thresholds because $\frac{\lvert\mathbf{a}_{i}^{H}\mathbf{x}\rvert}{\sqrt{\lvert\mathbf{a}_{i}^{H}\mathbf{z}^{(t)}\rvert^{2}+\mu_{(t)}^{2}}} < 2$ (it is bounded). Note that, if we considered  $\mu_{(t)}=0$ for all $t>0$ which is the SPARTA case, this implies that the gradient are given by
\begin{align}
\partial g(\mathbf{z}^{(t)},0) =  \frac{2}{m}\sum_{i=1}^{m}\left(1 - \frac{\lvert\mathbf{a}_{i}^{H}\mathbf{x}\rvert}{\lvert\mathbf{a}_{i}^{H}\mathbf{z}^{(t)}\rvert}\right)\mathbf{a}_{i}\mathbf{a}_{i}^{H}\mathbf{z}^{(t)}.
\label{eq:stafGradient}
\end{align}
Notice that \eqref{eq:stafGradient} could leads to excessively large size because of the term $\frac{\lvert\mathbf{a}_{i}^{H}\mathbf{x}\rvert}{\lvert\mathbf{a}_{i}^{H}\mathbf{z}^{(t)}\rvert}$, introducing bias in the update direction \cite{wang2016solving}. This fact is the main reason because \eqref{eq:stafGradient} (the Wirtinger gradient used in SPARTA) requires a truncation procedure in order to avoid a deviation in the update direction  \cite{wang2016solving,cande2}.

On the other hand, given the fact the proposed update direction in \eqref{eq:gradient} does not need truncation thresholds, then the proposed cost function $g(\mathbf{z},\mu)$ is locally smooth. In fact, Theorem \ref{theo:localSmo} establishes that the whole Wirtinger derivative $\partial g(\mathbf{z},\mu)$ in \eqref{eq:gradient} does not vary too much around of the curve of optimizers.

\begin{theorem}{\textit{(Local smoothness property \cite{candWir})}}
	The Wirtinger gradient defined in \eqref{eq:gradient} satisfies the following property
	\begin{align}
	\lVert \partial g(\mathbf{z},\mu)\rVert_{2} \leq \beta d_{r}(\mathbf{z},\mathbf{x}) +  \frac{\rho}{m}\sum_{k=1}^{m} \lvert \mathbf{a}^{H}_{k}\mathbf{h}\rvert,
	\end{align}
	where $\rho,\beta\in\mathbb{R}_{++}$ with probability at least $1-me^{-n/2}$ when $m\geq C(\epsilon_{0})n$ for some constant $C(\epsilon_{0})$ depending on $\epsilon_{0}>0$, and $\mathbf{h} = \mathbf{x} - e^{-j\theta(z)}\mathbf{z}$ with $\theta(z) = \argmin_{\theta\in[0,2\pi) }\lVert \mathbf{x} -e^{-j\theta}\mathbf{z} \rVert_{2}$.
	\label{theo:localSmo}
\end{theorem}

\begin{proof}
	The proof of the Theorem can be found in Appendix E.
\end{proof}

Finally, considering the result in Theorem \ref{theo:localSmo} we have that the local smoothness property it is preserved for the whole Wirtinger derivative $\partial g(\mathbf{z},\mu)$. In contrast, for those methods such as SPARTA that truncates the update direction, the local smoothness property it is preserved just for a piece of the direction update, introducing an important deviation of their search directions  \cite{wang2016sparse}, which reduces its performance to solve the phase retrieval problem as illustrated in Section \ref{sec:result}.

\section{Simulations and Results}
\label{sec:result}
In this section, the evaluation of the performance of the proposed method relative to SparseAltMinPhase (SAMP) \cite{netrapalli2013phase}, sparse WF (SWF) \cite{yuan2017phase} and Sparse Truncated Amplitude flow (SPARTA) \cite{wang2016sparse} is presented. All parameters for the implementation of each algorithm are their own suggested values in \cite{netrapalli2013phase,yuan2017phase,wang2016sparse}, respectively. The performance metric used is the $\textup{ relative error} := d_{r}(\mathbf{z,x})/\lVert \mathbf{x}\rVert_{2}; $ where $d_{r}(\mathbf{z,x}):= \min_{\theta \in [0,2\pi)} \lVert\mathbf{z}e^{-j\theta} - \mathbf{x}\rVert_{2}$ is the Euclidean distance modulo a global unimodular constant between two complex vectors, and $d_{r}(\mathbf{z,x}) = \min \lVert \mathbf{z \pm x}\rVert_{2}$ for the real case. We also evaluate the performance with the empirical success rate among 100 trial runs. For each trial, $1000$ iterations for all algorithms are employed. We declare that a trial is successful when the returned estimate incurs a relative error less than $10^{-5}$. All simulations are implemented in Matlab 2017a on an Intel Core i7 3.41Ghz CPU and 32 GB of RAM.

Six different tests are performed: the first assumes that the sparsity $k$ is known, the second and third consider that the sparsity is unknown, the fourth determines how the sparsity affects the ability of the methods to solve the sparse PR problems, the fifth considers the presence of noise, and finally the sixth scenario evaluates the reconstruction of a synthetic sparse signal. 

For all the experiments, the real signal is a Gaussian random vector generated as $\mathbf{x} \thicksim \mathcal{N}(0,\mathbf{I}_{1000})$ and the sampling vectors $\mathbf{a}_{i} \thicksim \mathcal{N}(0,\mathbf{I}_{1000})$ for $i=1,...,m$. For the complex Gaussian case $\mathbf{x} \thicksim \mathcal{N}(0,\mathbf{I}_{1000})+ j\mathcal{N}(0,\mathbf{I}_{1000})$ and the sampling vectors $\mathbf{a}_{i} \thicksim \mathcal{N}(0,\frac{1}{2}\mathbf{I}_{1000})+j\mathcal{N}(0,\frac{1}{2}\mathbf{I}_{1000})$ for $i=1,...,m$. The default values of the parameters of Algorithm \ref{alg:smothing} were determined using a cross-validation strategy. They were fixed as $\tau=0.3$ and the variables $\gamma=0.9$, $\gamma_1=0.5$, $\mu_{(0)}=30$ and $T=1000$.

\subsection*{Test 1: Known Sparsity}
The first experiment analyzes the sampling complexity under a noiseless real and complex Gaussian model, assuming that the sparsity $k$ is known. Figure \ref{fig:test1} summarizes the attained empirical success rate in terms of the number of measurements, for all algorithms under analysis. For this test, the sparsity of the signal $\mathbf{x}$ is fixed as $k=10$, and the ratio between $m$ and $n$ ($i.e$ $m/n$) is varied from 0.1 to 3, with a step size of 0.1, for both the real and the complex cases. At each ratio $m/n$, we calculate the average over 100 tests. 

\begin{figure}[ht]
	\centering
	\includegraphics[width=0.9\textwidth]{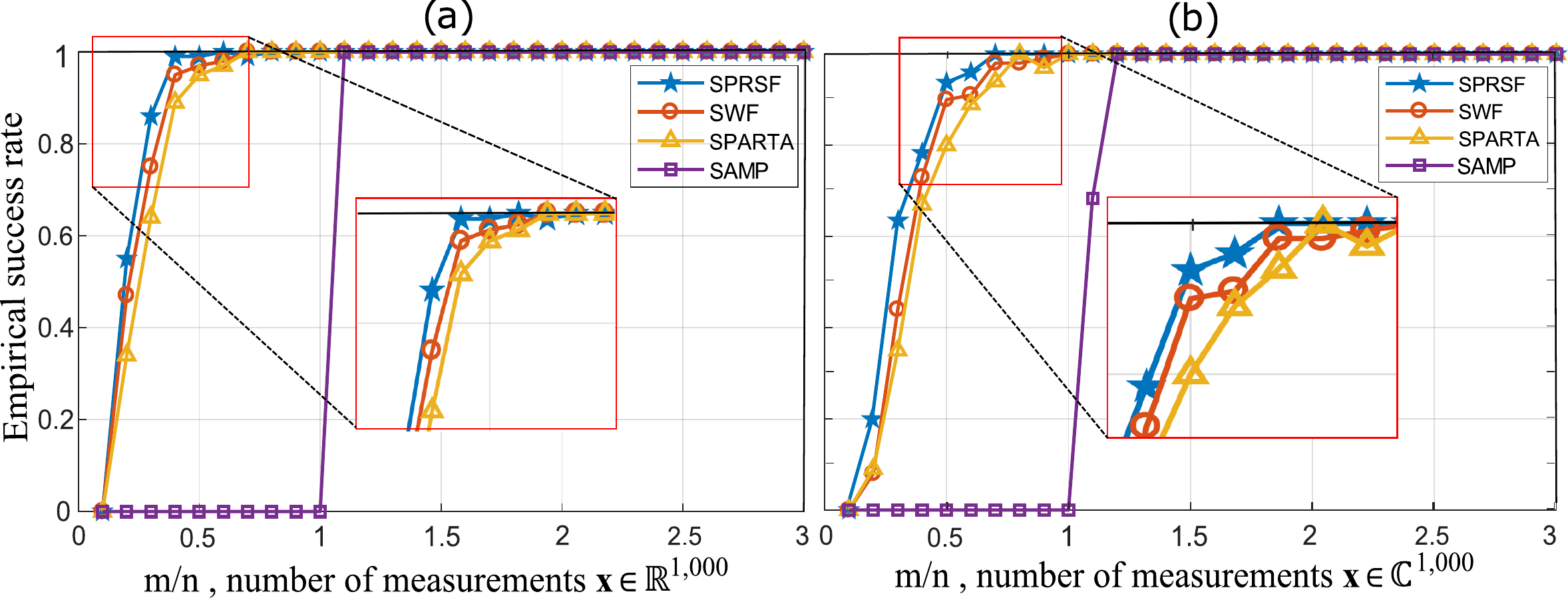}
	\caption{\footnotesize {Empirical success rate versus number of measurements for $n = 1000$, known sparsity $k=10$ and $m/n$ with a step size of 0.1 from 0.1 to 3. (a) Noiseless real-valued Gaussian model for $\mathbf{x} \thicksim \mathcal{N}(0,\mathbf{I}_{1000})$ and $\mathbf{a}_{i} \thicksim \mathcal{N}(0,\mathbf{I}_{1000})$. (b) Noiseless complex-valued Gaussian model, with $\mathbf{x} \thicksim \mathcal{N}(0,\mathbf{I}_{n})+ j\mathcal{N}(0,\mathbf{I}_{n})$ and $\mathbf{a}_{i} \thicksim \mathcal{N}(0,\frac{1}{2}\mathbf{I}_{n})+j\mathcal{N}(0,\frac{1}{2}\mathbf{I}_{n})$. }}
	\label{fig:test1}
\end{figure}

The simulations in Fig. \ref{fig:test1} suggest that the proposed algorithm SPRSF requires less number of measurements to solve the sparse phase retrieval problem in comparison with the SWF, SPARTA and SAMP methods, for both the real and the complex cases. Moreover, notice that SPRSF achieves a success rate over $98\%$ when $m/n = 0.5$ for the real case and a success rate over 95\% when $m/n=0.6$ for the complex case. Further, SPRSF guarantees a perfect recovery from about $0.6n$ and $0.7n$ measurements for the real and complex cases, respectively. Therefore, these results show the effectiveness of the smoothing approximation scheme to solve the sparse phase retrieval problem.

\subsection*{Test 2: Unknown Sparsity Boundary}
In this experiment, we compare the ability of the methods to recover the signal $\mathbf{x}$ in terms of the sampling complexity, when the sparsity $k$ is unknown. Specifically, from Theorem \ref{theo:contrac}, it can be obtained that the sampling complexity of the SPRSF method is $\mathcal{O}(k^{2}\log(n))$. Now, suppose that there is no knowledge about the sparsity $k$. If we assume that the sparsity is $k=\sqrt{n}$, the sampling complexity is given by $\mathcal{O}(n\log(n))$, which is considered the limit value of the unknown $k$ when $k\ll n$ \cite{yuan2017phase}. Therefore, in this Test the sparsity of the signal $\mathbf{x}$ is fixed to $k=10$, but the experiments, in Fig. \ref{fig:test2}, assume the sparsity of the signal $\mathbf{x}$ is $\sqrt{n} \approx 32$, since $n=1000$.\vspace{-1em}
\begin{figure}[ht]
	\centering
	\includegraphics[width=0.9\textwidth]{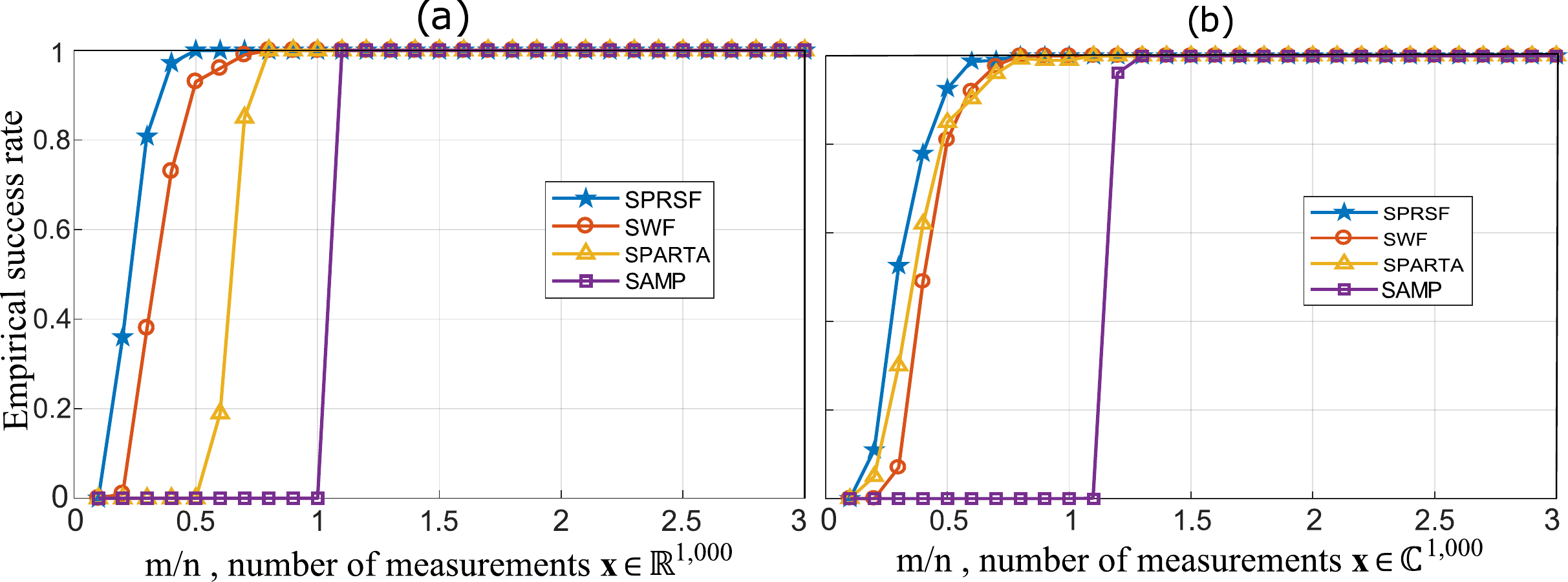}
	\caption{\footnotesize{Empirical success rate versus number of measurements for $n = 1,000$, $m/n$ with a step size of 0.1 from 0 to 3. The sparsity is assumed to be $k = \sqrt{n}\approx 32$ while the real sparsity is $k=10$. (a) Noiseless real-valued Gaussian model for $\mathbf{x} \thicksim \mathcal{N}(0,\mathbf{I}_{1000})$ and $\mathbf{a}_{i} \thicksim \mathcal{N}(0,\mathbf{I}_{1000})$. (b) Noiseless complex-valued Gaussian model, with $\mathbf{x} \thicksim \mathcal{N}(0,\mathbf{I}_{n})+ j\mathcal{N}(0,\mathbf{I}_{n})$ and $\mathbf{a}_{i} \thicksim \mathcal{N}(0,\frac{1}{2}\mathbf{I}_{n})+j\mathcal{N}(0,\frac{1}{2}\mathbf{I}_{n})$. }}
	\label{fig:test2}
\end{figure}

Notice that, SPRSF outperforms the other algorithms when the priori sparsity $k$ is not known correctly for both real and complex cases. Further, it can be observed that compared with Test 1 in Fig. \ref{fig:test1}, the superiority of the proposed method SPRSF with respect to SPARTA, SWF and SAMP, is more evident. Figure \ref{fig:test2} also shows that SPRSF attains a success rate of 80\% when $m/n = 0.3$ for the real case and a success rate of 90\% when $m/n=0.5$ for the complex case. Perfect recovery is attained from about $0.6n$ and $0.7n$ measurements for the real and the complex cases, respectively. 

It can be concluded that this second test suggests that the proposed smoothing approximation scheme overcomes its competitive alternatives when the sparsity is assumed different to its real value.

\subsection*{Test 3: Unknown Sparsity}

In this experiment, numerical simulations are conducted to analyze the ability of the methods to solve the sparse phase retrieval problem when the sparsity $k$ is completely unknown. For these simulations, the sparsity of the signal $\mathbf{x}$ was fixed as $k=10$ and since the sparsity is unknown, we range $\hat{k}$ from 35 to 180 for real and complex cases, with a step size of 5. At each $\hat{k}$, we calculate the average of the empirical success rate over 100 tests. We called the sparsity $\hat{k}$, the priori sparsity. The number of measurements $m$ was fixed to $m=n$. All these numerical tests are summarized in Fig. \ref{fig:test3}. We omitted the SAMP simulations in Fig. \ref{fig:test3}, since from Fig. \ref{fig:test1} it can be noticed that SAMP cannot solve the sparse PR problem when the sparsity $k$ is known and the number of measurements $m=n$.
\begin{figure}[ht]
	\centering
	\includegraphics[width=0.9\textwidth]{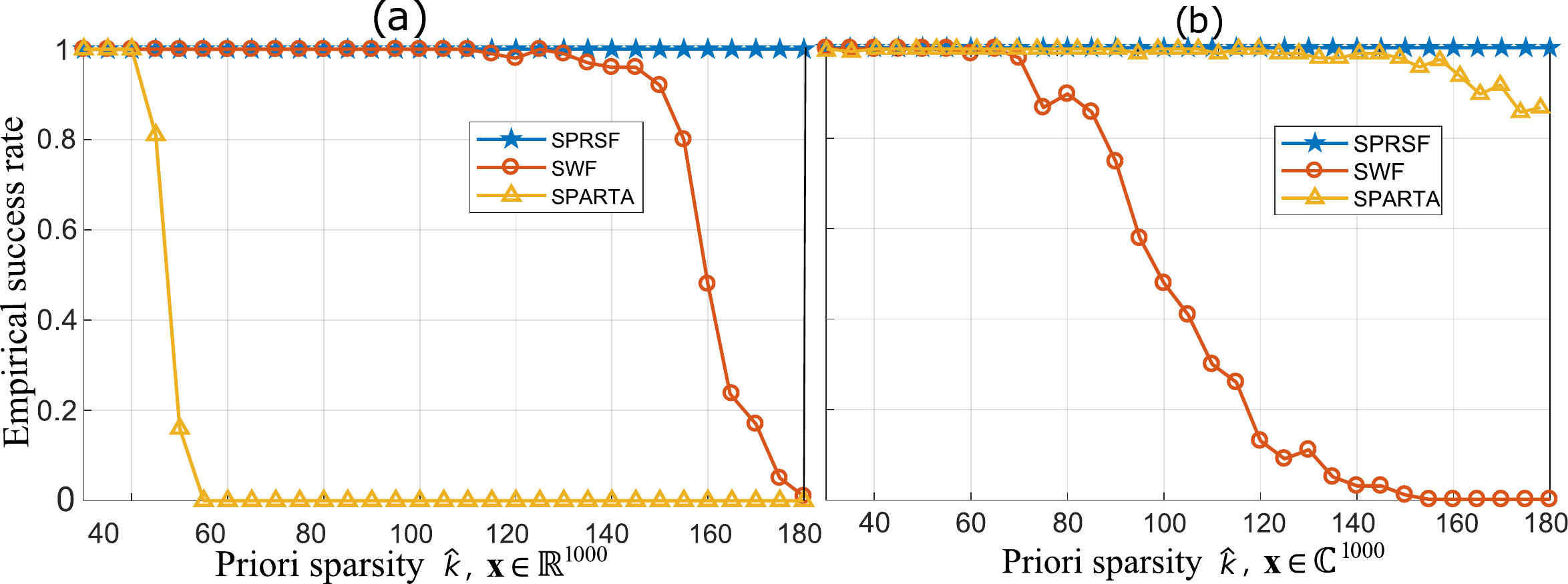}
	\caption{\footnotesize{Empirical success rate versus number of measurements for n = $1000$, $m/n = 1$, where the real sparsity is $k=10$. The priory sparsity $\hat{k}$ was ranged from 35 to 180, with a step size of 5. (a) Noiseless real-valued Gaussian model with $\mathbf{x} \thicksim \mathcal{N}(0,\mathbf{I}_{1000})$ and $\mathbf{a}_{i} \thicksim \mathcal{N}(0,\mathbf{I}_{1000})$. (b) Noiseless complex-valued Gaussian model with $\mathbf{x} \thicksim \mathcal{N}(0,\mathbf{I}_{n})+ j\mathcal{N}(0,\mathbf{I}_{n})$ and $\mathbf{a}_{i} \thicksim \mathcal{N}(0,\frac{1}{2}\mathbf{I}_{n})+j\mathcal{N}(0,\frac{1}{2}\mathbf{I}_{n})$. }}
	\label{fig:test3}
\end{figure}\vspace{-1em}

From Fig. \ref{fig:test3} it can be observed that the proposed method SPRSF overcomes its competing alternatives because it guarantees perfect recovery when the sparsity $k$ of the signal $\mathbf{x}$ is completely unknown. Further, notice that SPARTA cannot recover the signal without prior knowledge about the sparsity from a priori sparsity $\hat{k}=55$ and $\hat{k}=140$ for the real and complex cases, respectively, when the sparsity is $k=10$. Also, it can be concluded that SWF is superior to SPARTA for the real case, but SWF cannot recover the sparse signal from a priori sparsity $\hat{k}\geq 155$. However, for the complex case SPARTA exhibits a better performance than SWF, because SPARTA cannot always recover the signal from a priority sparsity $\hat{k}\geq 150$.

In summary, by combining the results from Test 2 (Fig. \ref{fig:test2})  and Test 3 (Fig. \ref{fig:test3}), it can be concluded that SPRSF is highly superior to SPARTA, SAMP and SWF in recovering the sparse signal $\mathbf{x}$ when there is no prior knowledge about the sparsity $k$.\vspace{-1em}

\subsection*{Test 4: Different Values of Sparsity Analysis}
This section shows numerical simulations to determine the effect of different sparsity values on the performance of SAMP, SPARTA, SWF and SPRSF. For these experiments we fixed the number of measurements $m=1.5n$ with $n=1000$ and the sparsity of the signal varying from $10$ to $100$ with a step size of 5. In these cases, we assume that the sparsity $k$ is known. All the numerical results are summarized in Fig. \ref{fig:test4}.\vspace{-1em}
\begin{figure}[H]
	\centering
	\includegraphics[width=0.9\textwidth]{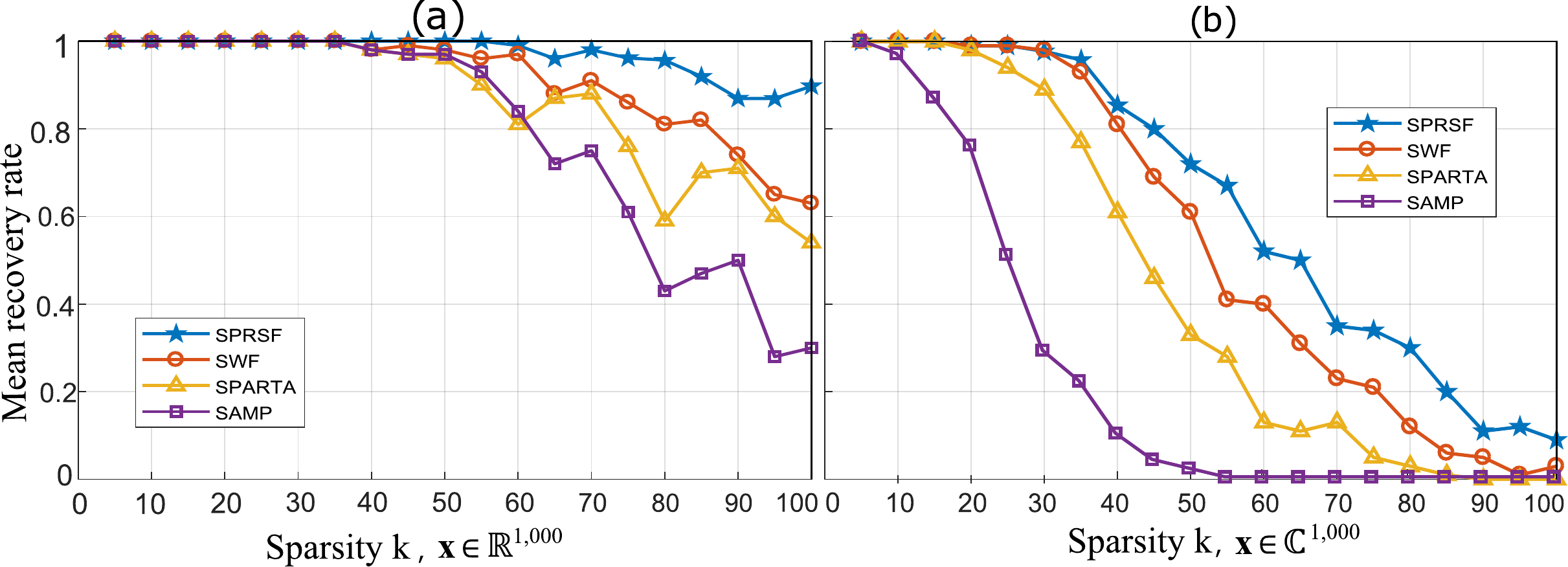}
	\caption{\footnotesize{Empirical success rate versus sparsity $k$ ranged from 10 to 100 with a step size of 5, $n = 1000$, $m/n = 1.5$. (a) Noiseless real-valued Gaussian model for $\mathbf{x} \thicksim \mathcal{N}(0,\mathbf{I}_{1000})$ and $\mathbf{a}_{i} \thicksim \mathcal{N}(0,\mathbf{I}_{1000})$. (b) Noiseless complex-valued Gaussian model, with $\mathbf{x} \thicksim \mathcal{N}(0,\mathbf{I}_{n})+ j\mathcal{N}(0,\mathbf{I}_{n})$ and $\mathbf{a}_{i} \thicksim \mathcal{N}(0,\frac{1}{2}\mathbf{I}_{n})+j\mathcal{N}(0,\frac{1}{2}\mathbf{I}_{n})$.}}
	\label{fig:test4}
\end{figure}\vspace{-1em}

Figure \ref{fig:test4} shows that the SPRSF method is superior to the SAMP, SPARTA and SWF algorithms, for both real and complex cases, since SPRSF can solve the sparse phase retrieval problem for signals with larger sparsity values, as opposed to its competitive alternatives. Also, it can be concluded that SPRSF has a mean recovery rate of about 75\% and 12\% when the sparsity is $k=100$ for the real and complex cases, respectively.\vspace{-1em}

\subsection*{Test 5: Noise Corruption Analysis}
Numerical tests are conducted to demonstrate the robustness of SPRSF to noise corruption. These simulations are performed under the noisy real/complex valued Gaussian model $\hat{y}_{i} = \vert \mathbf{a}_{i}^H \mathbf{x} \vert + \eta_{i}$. The noisy data was generated as $q_{i} = \hat{y}_{i}$ wit a signal to noise ratio (SNR) ranging from 5dB to 70dB. The number of measurements was fixed as $m=1.5n $ and the sparsity as $k=10$. The results in Fig. \ref{fig:test5} are the average of the relative error metric $d_{r}(\mathbf{z,x})/\lVert \mathbf{x}\rVert_{2}$  of 100 tests for each SNR value.\vspace{-1em}
\begin{figure}[H]
	\centering
	\includegraphics[width=0.9\textwidth]{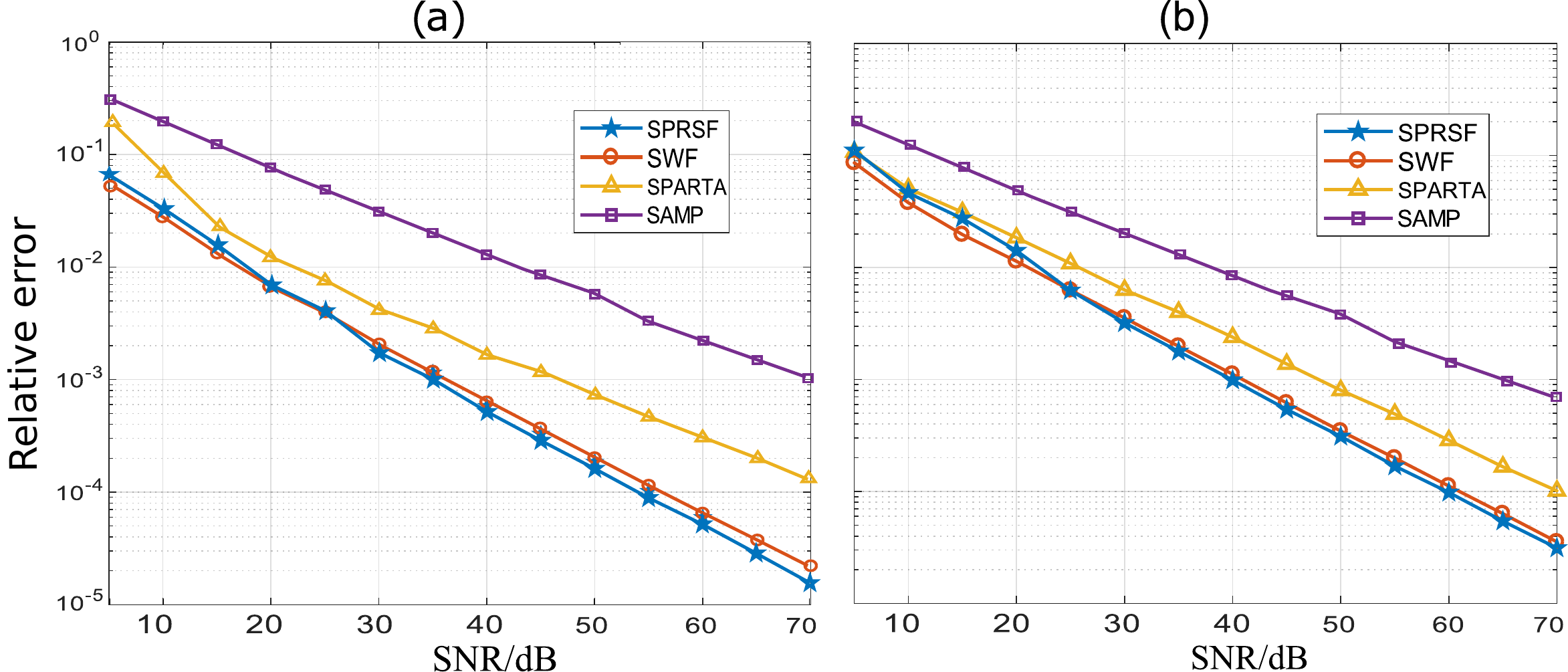}
	\caption{\footnotesize{Mean of 100 NMSE test for different values of  Gaussian white noise from 5dB to 70dB of SNR. (a) Noisy real-valued Gaussian model. (b) Noisy complex-valued Gaussian model.}}
	\label{fig:test5}
\end{figure}

From Fig. \ref{fig:test5} it can be observed that SWF attains a slightly better performance in solving the sparse phase retrieval problem, compared with SPRSF for the real and complex cases, in high-noise scenarios $0< SNR \leq 20$. However, when the noise level decreases, the proposed method overcomes that of SWF for both cases. Further, for the real and complex cases, the results show that SPRSF exhibits a better performance compared with its competitive SPARTA and SAMP alternatives for all values of noise.

\subsection*{Test 6: Speed of Convergence}
Simulations are conducted to compare the speed of convergence in absence of noise, under the limit case $m=n$ for both real and complex cases. The sparsity of the signal was fixed as $k=10$ and the priori sparsity as $\hat{k}=32$. Table \ref{tab:time} reports the number of iterations and the time cost required by all the algorithms to achieve a relative error of $10^{-14}$, averaged over 100 successful trials. In Table \ref{tab:time}, the optimal value of each column is shown in bold and the second-best result is underlined.\vspace{-1.5em}
\begin{center}
	\begin{table}[H]
		\caption{Comparison of iteration count and time cost among algorithms}
		\centering
		\begin{tabular}{|c|c|c|c|c|c|}
			\hline
			\textbf{Algorithms}& \multicolumn{2}{|c|}{\textbf{Real Case}}& \multicolumn{2}{|c|}{\textbf{Complex Case}}\\
			\hline
			& Iterations & Time (s) & Iterations & Time (s) \\
			\hline
			SPRSF & \textbf{85} & \underline{0.1061} & \textbf{103} & \textbf{0.145} \\
			\hline
			SPARTA & \underline{125} & \textbf{0.093} & \underline{128} & \underline{0.3945} \\
			\hline
			SWF & 243  & 5.1823 & 728 & 14.881 \\
			\hline
		\end{tabular}
		\label{tab:time}
	\end{table}
	\vspace{-2.5em}
\end{center}

From Table \ref{tab:time} it can be observed that SPRSF is the second best algorithm in terms of computational complexity in the real case, over all methods under analysis. However, for the complex case, SPRSF is the fastest to converge to the solution compared with SPARTA and SWF. On the other hand, we omitted the SAMP simulations in Table \ref{tab:time}, since from Fig. \ref{fig:test1} it can be noticed that SAMP cannot solve the sparse PR problem when the sparsity $k$ is known and the number of measurements $m=n$.

\subsection*{Test 7: Reconstructions}
Finally, to test the performance of the proposed algorithm on synthetic data, a random sparse signal $\mathbf{x} \in \mathbb{R}^{1000}/\mathbb{C}^{1000}$ is employed as illustrated in Fig. \ref{fig:res1}(a). The sparsity of the signal is $k=10$ and the number of measurements is fixed as $m=n$. The sampling vectors were generated as $\mathbf{a}_{i} \thicksim \mathcal{N}(0,\mathbf{I}_{1000})$ for $i=1,...,m$. The different analyzed algorithms were used to reconstruct the signal assuming a priori sparsity with value $\hat{k}=180$. The obtained reconstructions are shown from Fig. \ref{fig:res1}(b) to Fig. \ref{fig:res1}(d).
\begin{figure}[ht]
	\centering
	\includegraphics[width=0.8\textwidth]{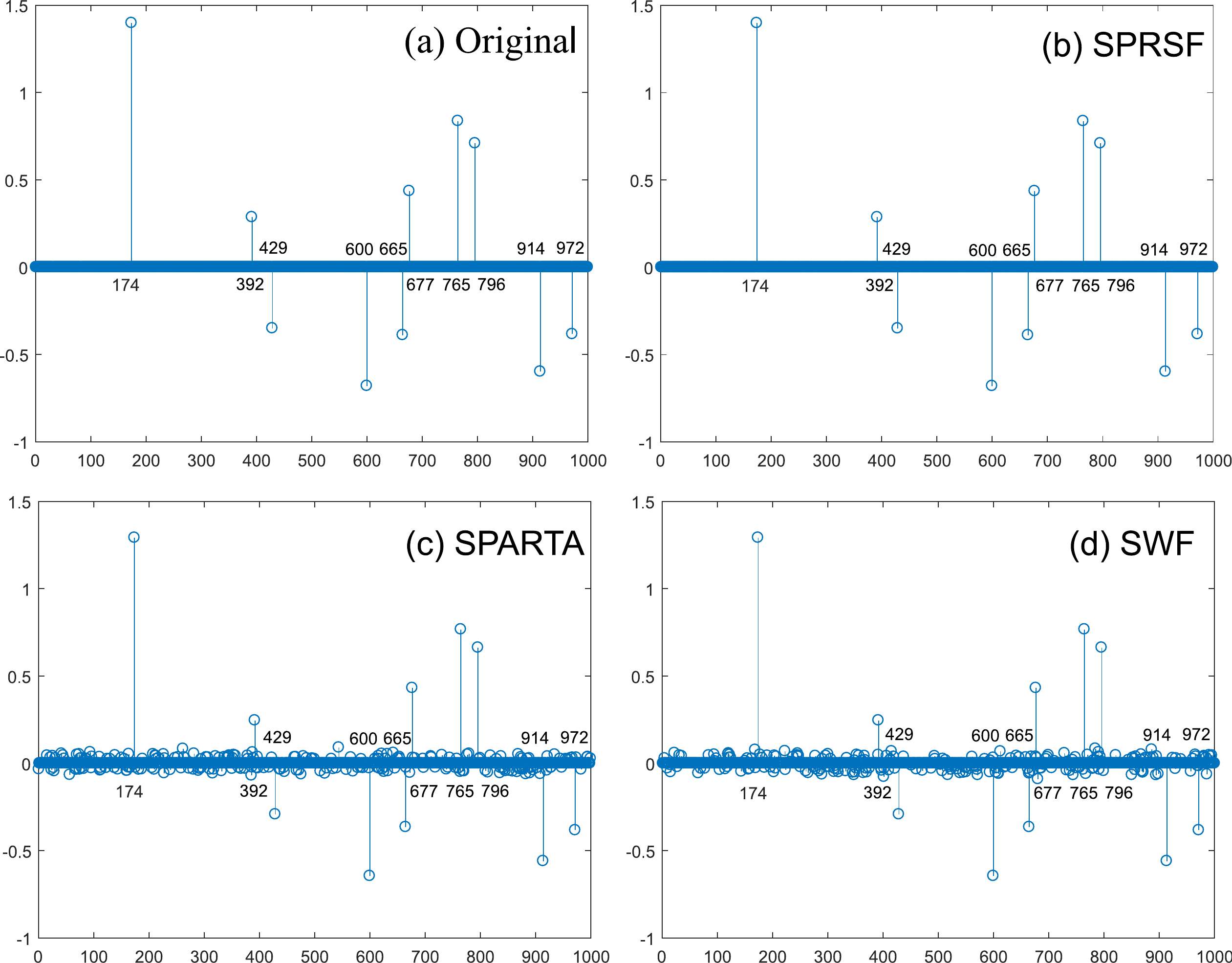}
	\caption{\footnotesize{(a) Original signal with sparsity $k=10$. Reconstructed signal with a priority sparsity $\hat{k}=180$ by (b) SPRSF, (c) SPARTA, and (d) SWF.} }
	\label{fig:res1}
\end{figure}

From Fig. \ref{fig:res1}(b) it can be observed that the proposed method SPRSF can recover perfectly the signal. In contrast, notice that the SPARTA and SWF methods can determine the real support of the original signal, but they also reconstructed nonzero values in positions in which the original signal has zero elements. Then, these numerical results suggest that the proposed method is highly superior to its competitive alternatives to solve the sparse phase retrieval problem when the sparsity is unknown.

\section{Conclusion}
This paper presented the Sparse Phase Retrieval Algorithm via Smoothing Function (SPRSF) to solve the sparse phase retrieval problem. SPRSF is an iterative algorithm where the update step is obtained by a hard thresholding over a gradient descent direction. Also, numerical experiments show an improvement of the SPRSF method in terms of sampling complexity, since it requires less number of measurements when compared to state of art methods such as SAMP, SPARTA and SWF. Moreover, the results also show the ability of the proposed method to recover the signal of interest even when the sparsity is unknown. Furthermore, the SPRSF algorithm guarantees perfect recovery when there is not prior knowledge of the sparsity. Moreover, SPRSF attains a higher mean recovery rate in comparison with the analyzed state of art methods, when the sparsity varies for the real and complex cases. Additionally, the SPRSF method attains a higher performance under a noisy model with respect to SAMP, SPARTA and SWF, even in highly noisy scenarios.


\section{Appendix A: Proof Lemma  \ref{lem:locallyLipschi}}
\begin{proof}
	To prove the lemma, we proceed to show first that for all $i\in\{1,\cdots,m\}$ the functions $f_{i}(\cdot)$ in \eqref{eq:traditional} are Lipschitz continuous. Let $\mathbf{w}_{1},\mathbf{w}_{2}\in \mathbb{C}^{n}$ be two different vectors such that
	\begin{equation}
	\lvert f_{i}(\mathbf{w}_{1})-f_{i}(\mathbf{w}_{2})\rvert = \lvert \lvert \langle \mathbf{a}_{i},\mathbf{w}_{1} \rangle \rvert - \lvert \langle \mathbf{a}_{i},\mathbf{w}_{2} \rangle \rvert \rvert.
	\label{eq:defLips}
	\end{equation}
	By using the triangle inequality on the right hand side term of \eqref{eq:defLips}, one can write
	\begin{equation}
	\lvert \lvert \langle \mathbf{a}_{i},\mathbf{w}_{1} \rangle \rvert - \lvert \langle \mathbf{a}_{i},\mathbf{w}_{2} \rangle \rvert \rvert \leq \lvert  \langle e^{-j\theta}\mathbf{w}_{1},\mathbf{a}_{i} \rangle - \langle \mathbf{w}_{2},\mathbf{a}_{i} \rangle \rvert,
	\label{eq:lipsTrian}
	\end{equation}
	for any $\theta\in [0,2\pi)$. Using the fact that $\langle \mathbf{w},\mathbf{a}_{i} \rangle=\mathbf{a}_{i}^{H}\mathbf{w}$ and from \eqref{eq:defLips} and \eqref{eq:lipsTrian}, it can be expressed that
	\begin{equation}
	\begin{array}{ll}
	\lvert f_{i}(\mathbf{w}_{1})-f_{i}(\mathbf{w}_{2})\rvert &\leq \lvert e^{-j\theta}\left(\mathbf{a}_{i}^{H}\mathbf{w}_{1}\right) - \left(\mathbf{a}_{i}^{H}\mathbf{w}_{2}\right) \rvert\\
	&\leq \lvert \mathbf{a}_{i}^{H}\left(e^{-j\theta}\mathbf{w}_{1}-\mathbf{w}_{2}\right) \rvert.
	\end{array}
	\label{eq:lipsTrian1}
	\end{equation}
	By definition $\mathbf{a}_{i}^{H}\mathbf{w}=\sum_{l=1}^{n} \left(\overline{\mathbf{a}}_{i}\right)_{l}\left(\mathbf{w}\right)_{i}$, where $\left(\overline{\mathbf{a}}_{i}\right)_{l}$ is the $l$-th conjugate component of $\mathbf{a}_{i}$ and, $\left(\mathbf{w}\right)_{i}$ is the $i$-th element of $\mathbf{w}$. Then, using the triangle inequality, \eqref{eq:lipsTrian1} can be rewritten as
	\begin{equation}
	\begin{array}{ll}
	\lvert f_{i}(\mathbf{w}_{1})- f_{i}(\mathbf{w}_{2})\rvert &\leq \lvert \sum_{l=1}^{n} \left(\overline{\mathbf{a}}_{i}\right)_{l}\left(e^{-j\theta}\mathbf{w}_{1}-\mathbf{w}_2\right)_{l} \rvert\\
	&\leq \sum_{l=1}^{n} \lvert \left(\mathbf{a}_{i}\right)_{l}\rvert \lvert \left(e^{-j\theta}\mathbf{w}_{1}- \mathbf{w}_2\right)_{l}\rvert\\
	&\leq a^{i}_{*}\sum_{l=1}^{n} \lvert \left(e^{-j\theta}\mathbf{w}_{1}- \mathbf{w}_2\right)_{l}\rvert\\
	&\leq a^{i}_{*}\lVert e^{-j\theta}\mathbf{w}_{1}- \mathbf{w}_2\rVert_{1},
	\end{array}
	\label{eq:lipsL1}
	\end{equation}
	where $a^{i}_{*}=\max\{\lvert \left(\mathbf{a}_{i}\right)_{l}\rvert: l=1,\cdots,n\}$ and $\lVert \cdot \rVert_{1}$ is the $\ell_{1}$ norm. Since $\ell_{1}$ and $\ell_{2}$ are equivalent norms, there exist a constant $\rho\in \mathbb{R}_{++}$ such that $\lVert \mathbf{w} \rVert_{1} \leq \rho \lVert \mathbf{w} \rVert_{2}$ for all $\mathbf{w}\in \mathbb{R}^{n}/\mathbb{C}^{n}$ \cite{candes2008introduction}. Thus, \eqref{eq:lipsL1} becomes
	\begin{equation}
	\begin{array}{ll}
	\lvert f_{i}(\mathbf{w}_{1})- f_{i}(\mathbf{w}_{2})\rvert &\leq a^{i}_{*}\lVert e^{-j\theta}\mathbf{w}_{1}- \mathbf{w}_2\rVert_{1}\\
	&\leq \left(a^{i}_{*}\rho\right)\lVert e^{-j\theta}\mathbf{w}_{1}- \mathbf{w}_2\rVert_{2}.
	\end{array}
	\label{eq:lipsFinal}
	\end{equation}
	Notice that, for the $i.i.d.$ Gaussian vectors $\mathbf{a}_{k}$, $a_{*}^{k}=\lVert \mathbf{a}_{k}\rVert_{\infty}\leq \sqrt{2.3n}$ holds with probability at least $1-me^{-n/2}$ \cite{wang2016solving}. Further, taking the value of $\theta$ that minimizes the term $\lVert e^{-j\theta}\mathbf{w}_{1}-\mathbf{w}_{2} \rVert_{2}$, \eqref{eq:lipsFinal} can be written as
	\begin{equation}
	\lvert f_{k}(\mathbf{w}_{1})- f_{k}(\mathbf{w}_{2})\rvert \leq \left(\sqrt{2.3n}\rho\right) d_{r}(\mathbf{w}_{1},\mathbf{w}_2).
	\label{eq:lipsFinal1}
	\end{equation}
	Therefore, from \eqref{eq:lipsFinal1} it can be concluded that each $f_{k}(\cdot)$ is a Lipschitz continuous function with constant $L_{k}=\sqrt{2.3n}\rho$ with probability at least $1-me^{-n/2}$. Further, the function $\frac{1}{\sqrt{m}}\left(f_{k}(\mathbf{x})-q_{k}\right)$ in \eqref{eq:traditional} is also Lipschitz continuous with constant $\sqrt{\frac{2.3n}{m}}\rho$ with probability exceeding $1-me^{-n/2}$, because the term $q_{k}$ can be considered as a constant \cite{eriksson2013applied}.
	
	On the other hand, take any $\mathbf{w}\in \mathbb{C}^{n}$ and define $\mathcal{U}=\{\mathbf{z}\in \mathbb{C}^{n}: d_{r}(\mathbf{z},\mathbf{w})<\epsilon \}$ for $\epsilon>0$. Note that $\mathcal{U}$ is the neighborhood of $\mathbf{w}$ and also $\mathcal{U}$ is a bounded set because $\lVert \mathbf{z}\rVert_{2}\leq \lVert \mathbf{w}\rVert_{2}+\epsilon < \infty$, for all $\mathbf{z}\in \mathcal{U}$. Thus, given the fact that $\mathcal{U}$ is a bounded set and each function $\frac{1}{\sqrt{m}}\left(f_{k}(\mathbf{x})-q_{k}\right)$ is a Lipschitz continuous function, then $\frac{1}{m}\left(f_{k}(\mathbf{x})-q_{k}\right)^{2}$ restricted to the set $\mathcal{U}$ is a Lipschitz continuous function \cite{eriksson2013applied} with probability at least $1-me^{-n/2}$. Hence, since $f(\mathbf{x})$ defined in \eqref{eq:traditional} is a sum of Lipschitz continuous functions in the set $\mathcal{U}$, then $f(\mathbf{x})$ is a Lipschitz continuous function in $\mathcal{U}$. Thus, it can be concluded that $f(\mathbf{x})$ is locally Lipschitz continuous according to Definition \ref{def:lipschitz1} with probability at least $1-me^{-n/2}$.
	\end{proof}

\section*{Appendix B: Proof of Theorem \ref{theo:contrac}}
\label{app:proofTheorem1}
\begin{proof}
	Let $\mathbf{h}^{(t)} = \mathbf{x} - e^{-j\theta_{(t)}}\mathbf{z}^{(t)}$ with $\mathbf{z}^{(t)}$ and $\theta_{(t)} = \argmin_{\theta\in[0,2\pi) }\lVert \mathbf{x} -e^{-j\theta}\mathbf{z}^{(t)} \rVert_{2}$. Also, define
	\begin{align}
		\label{eq:definition1}
		\mathbf{d}^{(t)}&=\mathbf{z}^{(t)} - \tau\partial g(\mathbf{z}^{(t)},\mu_{(t)})\\ \nonumber
		&=\mathbf{z}^{(t)}-\frac{2\tau}{m}\sum_{i=1}^{m}\left(\sqrt{\lvert\mathbf{a}_{i}^{H}\mathbf{z}^{(t)}\rvert^{2}+\mu_{(t)}^{2}}-q_{i}\right)\frac{\mathbf{a}_{i}^{H}\mathbf{z}^{(t)}}{\sqrt{\lvert\mathbf{a}_{i}^{H}\mathbf{z}^{(t)}\rvert^{2}+\mu_{(t)}^{2}}}\mathbf{a}_{i},
	\end{align}
	for $t=0,1,\dots,\infty$, which stands for the prior estimate to the hard thresholding operation in Algorithm \ref{alg:smothing}, Line 4. Let $\Theta_{(t+1)}=S_{(t+1)}\cup S^{*}$ be a set where $S_{(t+1)}$ is the support of $\mathbf{z}^{(t+1)}$, and $S^{*}$ is the support of the real solution $\mathbf{x}$. The reconstruction error $\mathbf{h}^{(t+1)}$ is supported on the set $\Theta_{(t+1)} :=S^{*}\cup S_{(t+1)}$; likewise, $\mathbf{h}^{(t)}$ is supported on $\Theta_{(t)} :=S^{*}\cup S_{(t)}$. Moreover, the difference between $\Theta_{(t)}$ and $\Theta_{(t+1)}$ can be defined as $\Theta_{(t)}\setminus \Theta_{(t+1)}$, which consists of all elements of $\Theta_{(t)}$ that are not elements of $\Theta_{(t+1)}$. It is then clear that $|S^{*}| = |S_{(t)}|=k$, $|\Theta_{(t)}|\leq 2k$, and $|\Theta_{(t)}\setminus \Theta_{(t+1)}|\leq 2k$ as well as $|\Theta_{(t)}\cup \Theta_{(t+1)}|\leq 3k$ for all $t\geq 0$. When using these sets as subscript, for instance, $\mathbf{d}^{(t)}_{\Theta_{(t)}}$, we mean vectors formed by setting to zero all but those elements from the vector other than those in the set. 
	
	Note that, by definition of $d_{r}(\cdot,\cdot)$ we have that
	\begin{align}
	d_{r}(\mathbf{z}^{(t+1)}_{\Theta_{(t+1)}},\mathbf{x}_{\Theta_{(t+1)}}) = \min_{\theta\in[0,2\pi) }\lVert \mathbf{x}_{\Theta_{(t+1)}} - e^{-j\theta}\mathbf{z}^{(t+1)}_{\Theta_{(t+1)}} \rVert_{2}\leq \lVert \mathbf{x}_{\Theta_{(t+1)}} - e^{-j\theta_{(t)}}\mathbf{z}^{(t+1)}_{\Theta_{(t+1)}}\rVert_{2}.
	\label{eq:distanceMin}
	\end{align}
	Then, notice that by using the triangle inequality, one can write that
	\small{\begin{align}
		\label{eq:inequality1}
		\lVert \mathbf{x}_{\Theta_{(t+1)}} -e^{-j\theta_{(t)}}\mathbf{z}^{(t+1)}_{\Theta_{(t+1)}}\rVert_{2} &= \lVert \mathbf{x}_{\Theta_{(t+1)}} -e^{-j\theta_{(t)}}\mathbf{d}^{(t+1)}_{\Theta_{(t+1)}} +e^{-j\theta_{(t)}}\mathbf{d}^{(t+1)}_{\Theta_{(t+1)}} -e^{-j\theta_{(t)}}\mathbf{z}^{(t+1)}_{\Theta_{(t+1)}}\rVert_{2} \nonumber\\
		& \leq \lVert \mathbf{x}_{\Theta_{(t+1)}} -e^{-j\theta_{(t)}}\mathbf{d}^{(t+1)}_{\Theta_{(t+1)}}\rVert_{2} \nonumber\\
		&+ \lVert e^{-j\theta_{(t)}}\mathbf{z}^{(t+1)}_{\Theta_{(t+1)}} -e^{-j\theta_{(t)}}\mathbf{d}^{(t+1)}_{\Theta_{(t+1)}}\rVert_{2},
	\end{align}}\normalsize
	where in the last inequality the first term is the distance of $\mathbf{x}_{\Theta_{(t+1)}}$ to the estimate $\mathbf{d}^{(t+1)}_{\Theta_{(t+1)}}$ before hard thresholding, and the second is the distance between $\mathbf{d}^{(t+1)}_{\Theta_{(t+1)}}$ and its best $k$-approximation $\mathbf{z}^{(t+1)}_{\Theta_{(t+1)}}$ due to $|\Theta_{(t+1)}|\leq 2k$. The optimality of $\mathbf{z}^{(t+1)}_{\Theta_{(t+1)}}$ implies $\lVert  e^{-j\theta_{(t)}}\mathbf{z}^{(t+1)}_{\Theta_{(t+1)}} -e^{-j\theta_{(t)}}\mathbf{d}^{(t+1)}_{\Theta_{(t+1)}}\rVert_{2} \leq \lVert \mathbf{x}_{\Theta_{(t+1)}} -e^{-j\theta_{(t)}}\mathbf{d}^{(t+1)}_{\Theta_{(t+1)}}\rVert_{2}$. 
	
	Plugging the latter relationship into \eqref{eq:inequality1} yields
	\begin{equation}
	\lVert \mathbf{x}_{\Theta_{(t+1)}} -e^{-j\theta_{(t)}}\mathbf{z}^{(t+1)}_{\Theta_{(t+1)}}\rVert_{2} \leq 2\lVert \mathbf{x}_{\Theta_{(t+1)}} -e^{-j\theta_{(t)}}\mathbf{d}^{(t+1)}_{\Theta_{(t+1)}}\rVert_{2},
	\label{eq:normInequality}
	\end{equation}
	where the equality in \eqref{eq:inequality1} arises from restricting our analysis solely to the support $\Theta_{(t+1)}$ of $\mathbf{x}-e^{-j\theta_{(t)}}\mathbf{d}^{(t+1)}$. Then, considering \eqref{eq:definition1}, the vector $e^{-j\theta_{(t)}}\mathbf{d}^{(t)}$ can be rewritten as
	\begin{equation}
	e^{-j\theta_{(t)}}\mathbf{d}^{(t+1)} = e^{-j\theta_{(t)}}\mathbf{z}^{(t)}+\frac{2\tau}{m}\sum_{i=1}^{m}\left(\mathbf{a}_{i}^{H}\mathbf{h}^{(t)}+q_{i}\left(\frac{e^{-j\theta_{(t)}}\mathbf{a}_{i}^{H}\mathbf{z}^{(t)}}{\sqrt{\lvert\mathbf{a}_{i}^{H}\mathbf{z}^{(t)}\rvert^{2}+\mu_{(t)}^{2}}}-\frac{\mathbf{a}_{i}^{H}\mathbf{x} }{\lvert \mathbf{a}_{i}^{H}\mathbf{x} \rvert}\right)\right)\mathbf{a}_{i}.
	\label{eq:rewriD}
	\end{equation}
	\vspace{-0.1em}
	Combining \eqref{eq:normInequality} and \eqref{eq:rewriD} it can be obtained that
	\begin{align}
		\label{eq:hInequality}
		\frac{1}{2}\lVert \mathbf{h}^{(t+1)} \rVert_{2} \leq & \big\lVert \mathbf{x}_{\Theta^{(t+1)}} - e^{-j\theta_{(t)}}\mathbf{z}_{\Theta^{(t+1)}}^{(t)} - \frac{2\tau}{m} \sum_{i=1}^{m}\left(\mathbf{a}_{i}^{H}\mathbf{h}^{(t)}\right)\mathbf{a}_{i,\Theta{(t+1)}}  \\ \nonumber
		-& \frac{2\tau}{m} \sum_{i=1}^{m} \left(\frac{e^{-j\theta_{(t)}}\mathbf{a}_{i}^{H}\mathbf{z}^{(t)}}{\sqrt{\lvert\mathbf{a}_{i}^{H}\mathbf{z}^{(t)}\rvert^{2}+\mu_{(t)}^{2}}}-\frac{\mathbf{a}_{i}^{H}\mathbf{x} }{\lvert \mathbf{a}_{i}^{H}\mathbf{x} \rvert}\right)\lvert \mathbf{a}_{i}^{H}\mathbf{x} \rvert \mathbf{a}_{i,\Theta^{(t+1)}}  \big\rVert_{2}\\ \nonumber
		= & \big\lVert \mathbf{h}_{\Theta^{(t+1)}}^{(t)} - \frac{2\tau}{m} \sum_{i=1}^{m}\mathbf{a}_{i,\Theta^{(t+1)}}\mathbf{a}_{i,\Theta^{(t+1)}}^{H}\mathbf{h}_{\Theta^{(t+1)}}^{(t)} \\ \nonumber
		-& \frac{2\tau}{m} \sum_{i=1}^{m}\mathbf{a}_{i,\Theta{(t+1)}}\mathbf{a}_{i,\Theta^{(t)}\setminus\Theta^{(t+1)}}^{H}\mathbf{h}_{\Theta^{(t)}\setminus\Theta^{(t+1)}}^{(t)} \\ \nonumber
		-& \frac{2\tau}{m} \sum_{i=1}^{m} \left(\frac{e^{-j\theta_{(t)}}\mathbf{a}_{i}^{H}\mathbf{z}^{(t)}}{\sqrt{\lvert\mathbf{a}_{i}^{H}\mathbf{z}^{(t)}\rvert^{2}+\mu_{(t)}^{2}}}-\frac{\mathbf{a}_{i}^{H}\mathbf{x} }{\lvert \mathbf{a}_{i}^{H}\mathbf{x} \rvert}\right)\lvert \mathbf{a}_{i}^{H}\mathbf{x} \rvert \mathbf{a}_{i,\Theta^{(t+1)}}  \big\rVert_{2},
	\end{align}
	where the equality follows from re-writing $\mathbf{a}_{i}^{H}\mathbf{h}^{(t)}=\mathbf{a}_{i,\Theta^{(t)}}^{H}\mathbf{h}_{\Theta^{(t)}}^{(t)} = \mathbf{a}_{i,\Theta^{(t+1)}}^{H}\mathbf{h}_{\Theta^{(t+1)}}^{(t)} + \mathbf{a}_{i,\Theta^{(t)}\setminus\Theta^{(t+1)}}^{H}\mathbf{h}_{\Theta^{(t)}\setminus\Theta^{(t+1)}}^{(t)}$. Then, from \eqref{eq:hInequality} we have that 
	\begin{align}
		\label{eq:inequality2}
		\frac{1}{2}\lVert \mathbf{h}^{(t+1)} \rVert_{2} \leq & \overbrace{\left\lVert \mathbf{h}_{\Theta^{(t+1)}}^{(t)} - \frac{2\tau}{m} \sum_{i=1}^{m}\mathbf{a}_{i,\Theta^{(t+1)}}\mathbf{a}_{i,\Theta^{(t+1)}}^{H}\mathbf{h}_{\Theta^{(t+1)}}^{(t)} \right\rVert_{2}}^{v_{1}}\\ \nonumber
		+& \underbrace{\left\lVert \frac{2\tau}{m} \sum_{i=1}^{m}\mathbf{a}_{i,\Theta^{(t+1)}}\mathbf{a}_{i,\Theta^{(t)}\setminus\Theta^{(t+1)}}^{H}\mathbf{h}_{\Theta^{(t)}\setminus\Theta^{(t+1)}}^{(t)} \right\rVert_{2}}_{v_{2}} \\ \nonumber
		+& \underbrace{\left \lVert \frac{2\tau}{m} \sum_{i=1}^{m} \left(\frac{e^{-j\theta_{(t)}}\mathbf{a}_{i}^{H}\mathbf{z}^{(t)}}{\sqrt{\lvert\mathbf{a}_{i}^{H}\mathbf{z}^{(t)}\rvert^{2}+\mu_{(t)}^{2}}}-\frac{\mathbf{a}_{i}^{H}\mathbf{x} }{\lvert \mathbf{a}_{i}^{H}\mathbf{x} \rvert}\right)\lvert \mathbf{a}_{i}^{H}\mathbf{x} \rvert \mathbf{a}_{i,\Theta^{(t+1)}}  \right\rVert_{2}}_{v_{3}}.
	\end{align}
	
	Notice that from \eqref{eq:inequality2} it can be obtained that
	\begin{align}
	\label{eq:primerTerm}
	v_{1} =& \left\lVert \left(\mathbf{I}_{n} - \frac{2\tau}{m} \sum_{i=1}^{m}\mathbf{a}_{i,\Theta^{(t+1)}}\mathbf{a}_{i,\Theta^{(t+1)}}^{H}\right) \mathbf{h}_{\Theta^{(t+1)}}^{(t)} \right\rVert_{2} \\ \nonumber
	\leq & \left\lVert\mathbf{I}_{n} - \frac{2\tau}{m} \sum_{i=1}^{m}\mathbf{a}_{i,\Theta^{(t+1)}}\mathbf{a}_{i,\Theta^{(t+1)}}^{H} \right\rVert_{2\rightarrow 2}\lVert \mathbf{h}_{\Theta^{(t+1)}}^{(t)}\rVert_{2}\\ \nonumber
	\leq & \max\{1-2\tau\underline{\lambda},2\tau\overline{\lambda}-1 \}\lVert \mathbf{h}_{\Theta^{(t+1)}}^{(t)}\rVert_{2},
	\end{align}
	where $\lVert \cdot \rVert_{2\rightarrow 2}$ is the spectral norm and $\overline{\lambda},\underline{\lambda}>0$ are the largest and the smallest eigenvalues of $\frac{1}{m} \sum_{i=1}^{m}\mathbf{a}_{i,\Theta^{(t+1)}}\mathbf{a}_{i,\Theta^{(t+1)}}^{H}$, respectively. Then, by corollary 5.35 in \cite{vershynin2010introduction} it can be obtained that
	\begin{equation}
	\overline{\lambda} = \lambda_{max}\left(\frac{1}{m} \sum_{i=1}^{m}\mathbf{a}_{i,\Theta^{(t+1)}}\mathbf{a}_{i,\Theta^{(t+1)}}^{H}\right) \leq 1+\epsilon_{0},
	\label{eq:approximation1}
	\end{equation}
	with hight probability when $m\geq C(\epsilon_{0})2k$ for some constant $C(\epsilon_{0})$ depending on $\epsilon_{0}>0$. Moreover, by Lemma 5 in \cite{wang2016solving} we have that
	\begin{equation}
	\underline{\lambda} = \lambda_{min}\left(\frac{1}{m} \sum_{i=1}^{m}\mathbf{a}_{i,\Theta^{(t+1)}}\mathbf{a}_{i,\Theta^{(t+1)}}^{H}\right)\geq 1-\zeta_{1}-\epsilon_{1}
	\label{eq:approximation2}
	\end{equation} 
	when $m\geq C(\epsilon_{1})k$ for some constant $C(\epsilon_{1})$ depending on $\epsilon_{1}>0$. Taking the results in \eqref{eq:approximation1} and \eqref{eq:approximation2} into \eqref{eq:primerTerm} yields
	\begin{equation}
	v_{1}\leq \max\{ 1-2\tau(1-\zeta_{1}-\epsilon_{1}),2\tau(1+\epsilon_{0})-1 \}\lVert \mathbf{h}_{\Theta^{(t+1)}}^{(t)}\rVert_{2}.
	\label{eq:bound1}
	\end{equation}
	
	For the second term $v_{2}$ in \eqref{eq:inequality2}, fix any $\epsilon_{2}>0$. If the ratio number of measurements and unknowns $m/3k$, exceeds some sufficiently large constant, the next holds with probability of at least $1-2\exp\left(-c(\epsilon_{2})m\right)$ 
	\begin{align}
	\label{eq:secondTerm}
	v_{2} \leq& \left\lVert \frac{2\tau}{m} \sum_{i=1}^{m}\mathbf{a}_{i,\Theta^{(t+1)}}\mathbf{a}_{i,\Theta^{(t)}\setminus\Theta^{(t+1)}}^{H}\right\rVert_{2\rightarrow 2}\lVert \mathbf{h}_{\Theta^{(t)}\setminus\Theta^{(t+1)}}^{(t)}\rVert_{2}\\ \nonumber
	\leq& 2\tau \left\lVert \mathbf{I}_{n} - \frac{1}{m} \sum_{i=1}^{m}\mathbf{a}_{i,\Theta^{(t+1)}\cup \Theta^{(t)}}\mathbf{a}_{i,\Theta^{(t+1)}\cup \Theta^{(t)}}^{H}\right\rVert_{2\rightarrow 2}\lVert \mathbf{h}_{\Theta^{(t)}\setminus\Theta^{(t+1)}}^{(t)}\rVert_{2}\\ \nonumber
	\leq & 2\tau(\zeta_{2}+\epsilon_{2})\lVert \mathbf{h}_{\Theta^{(t)}\setminus\Theta^{(t+1)}}^{(t)}\rVert_{2},
	\end{align}
	in which the first inequality arises from the triangle inequality. The second inequality is obtained by Lemma 1 in \cite{blumensath2009iterative}. Similar to \eqref{eq:primerTerm}, the last inequality in \eqref{eq:secondTerm} is obtained by using corollary 5.35 in \cite{vershynin2010introduction} for some universal constants $c(\epsilon_{2})$ and $C(\epsilon_{2})$ such as $m\geq C(\epsilon_{2})2k$.
	
	Considering the last term $v_{3}$ in \eqref{eq:inequality2}, define $\mathbf{A}:=[\mathbf{a}_{1,\Theta^{(t+1)}},\cdots,\mathbf{a}_{m,\Theta^{(t+1)}}]$ and $\mathbf{b}^{(t)}:=[b^{(t)}_{1},\cdots,b_{m}^{(t)}]^{T}$ with $b_{i}^{(t)} = \left(\frac{e^{-j\theta_{(t)}}\mathbf{a}_{i}^{H}\mathbf{z}^{(t)}}{\sqrt{\lvert\mathbf{a}_{i}^{H}\mathbf{z}^{(t)}\rvert^{2}+\mu_{(t)}^{2}}}-\frac{\mathbf{a}_{i}^{H}\mathbf{x} }{\lvert \mathbf{a}_{i}^{H}\mathbf{x} \rvert}\right)\lvert \mathbf{a}_{i}^{H}\mathbf{x} \rvert$, for $i=1,\cdots,m$. Then, the $v_{3}$ term in \eqref{eq:inequality2} can be rewritten as
	\small{\begin{align}
		\label{eq:hterm}
		v_{3} = \left\lVert \frac{2\tau}{m}\mathbf{A}_{\Theta^{(t+1)}}^{T}\mathbf{b}^{(t)} \right\rVert_{2} \leq& 2\tau \left \lVert \frac{1}{\sqrt{m}} \mathbf{A}_{\Theta^{(t+1)}}^{T} \right\rVert_{2\rightarrow 2}\left\lVert \frac{1}{\sqrt{m}}\mathbf{b}^{(t)}\right\rVert_{2}\\ \nonumber
		\leq & 2\tau(1+\epsilon_{3})\left\lVert \frac{1}{\sqrt{m}}\mathbf{b}^{(t)}\right\rVert_{2},
		\end{align}}\normalsize
	where the second inequality is obtained by a standard matrix concentration result for any fixed $\epsilon_{3}>0$, with probability $1-2\exp(-c(\epsilon_{3})m)$, provided that $m\geq C(\epsilon_{3})k$, for some sufficiently large constant $C(\epsilon_{3})>0$.
	
	Notice that, from the definition of vector $\mathbf{b}^{(t)}$ it can be obtained that
	\begin{equation}
	\frac{1}{m}\left\lVert\mathbf{b}^{(t)}\right\rVert_{2}^{2} = \frac{1}{m} \sum_{i=1}^{m} \left\lvert\frac{e^{-j\theta_{(t)}}\mathbf{a}_{i}^{H}\mathbf{z}^{(t)}}{\sqrt{\lvert\mathbf{a}_{i}^{H}\mathbf{z}^{(t)}\rvert^{2}+\mu_{(t)}^{2}}}-\frac{\mathbf{a}_{i}^{H}\mathbf{x} }{\lvert \mathbf{a}_{i}^{H}\mathbf{x} \rvert}\right\rvert^{2} \lvert\mathbf{a}_{i}^{H}\mathbf{x}\rvert^{2}.
	\label{eq:normBt}
	\end{equation}
	
	Notice that from \eqref{eq:normBt} one can write that
	\begin{align}
		\left\lvert\frac{e^{-j\theta_{(t)}}\mathbf{a}_{i}^{H}\mathbf{z}^{(t)}}{\sqrt{\lvert\mathbf{a}_{i}^{H}\mathbf{z}^{(t)}\rvert^{2}+\mu_{(t)}^{2}}}-\frac{\mathbf{a}_{i}^{H}\mathbf{x} }{\lvert \mathbf{a}_{i}^{H}\mathbf{x} \rvert}\right\rvert \lvert\mathbf{a}_{i}^{H}\mathbf{x}\rvert &\leq  \lvert\mathbf{a}_{i}^{H}\mathbf{x}\rvert\left\lvert \frac{e^{-j\theta_{(t)}}\mathbf{a}_{i}^{H}\mathbf{z}^{(t)}}{\sqrt{\lvert\mathbf{a}_{i}^{H}\mathbf{z}^{(t)}\rvert^{2}+\mu_{(t)}^{2}}} - \frac{e^{-j\theta_{(t)}}\mathbf{a}_{i}^{H}\mathbf{z}^{(t)}}{\lvert\mathbf{a}_{i}^{H}\mathbf{x}\rvert}  \right\rvert \nonumber \\
		&+ \lvert\mathbf{a}_{i}^{H}\mathbf{x}\rvert \left\lvert \frac{e^{-j\theta_{(t)}}\mathbf{a}_{i}^{H}\mathbf{z}^{(t)}}{\lvert\mathbf{a}_{i}^{H}\mathbf{x}\rvert} - \frac{\mathbf{a}_{i}^{H}\mathbf{x}}{\lvert\mathbf{a}_{i}^{H}\mathbf{x} \rvert} \right\rvert \nonumber \\
		&\leq \left\lvert\sqrt{\lvert \mathbf{a}_{i}^{H}\mathbf{z}^{(t)} \rvert^{2}+\mu_{(t)}^{2}} -\lvert \mathbf{a}^{H}_{i}\mathbf{x}\rvert\right\rvert + \lvert \mathbf{a}_{i}^{H}\mathbf{h}^{(t)}\rvert
		\label{eq:normBt1}
	\end{align}
	in which the second inequality comes from the fact that
	\small{\begin{align}
	\lvert\mathbf{a}_{i}^{H}\mathbf{x}\rvert\left\lvert \frac{e^{-j\theta_{(t)}}\mathbf{a}_{i}^{H}\mathbf{z}^{(t)}}{\sqrt{\lvert\mathbf{a}_{i}^{H}\mathbf{z}^{(t)}\rvert^{2}+\mu_{(t)}^{2}}} - \frac{e^{-j\theta_{(t)}}\mathbf{a}_{i}^{H}\mathbf{z}^{(t)}}{\lvert\mathbf{a}_{i}^{H}\mathbf{x}\rvert} \right\rvert &\leq \frac{\lvert\mathbf{a}_{i}^{H}\mathbf{x}\rvert\lvert \mathbf{a}^{H}_{i}\mathbf{z}^{(t)}\rvert}{\lvert\mathbf{a}_{i}^{H}\mathbf{x}\rvert\sqrt{\lvert\mathbf{a}_{i}^{H}\mathbf{z}^{(t)}\rvert^{2}+\mu_{(t)}^{2}}} \left\lvert\sqrt{\lvert \mathbf{a}_{i}^{H}\mathbf{z}^{(t)} \rvert^{2}+\mu_{(t)}^{2}} -\lvert \mathbf{a}^{H}_{i}\mathbf{x}\rvert\right\rvert \nonumber\\
	&\leq \left\lvert\sqrt{\lvert \mathbf{a}_{i}^{H}\mathbf{z}^{(t)} \rvert^{2}+\mu_{(t)}^{2}} -\lvert \mathbf{a}^{H}_{i}\mathbf{x}\rvert\right\rvert.
	\label{eq:inequality4}
	\end{align}}\normalsize
	Then, from \eqref{eq:inequality4} it can be obtained that
	\begin{align}
	\left\lvert\sqrt{\lvert \mathbf{a}_{i}^{H}\mathbf{z}^{(t)} \rvert^{2}+\mu_{(t)}^{2}} -\lvert \mathbf{a}^{H}_{i}\mathbf{x}\rvert\right\rvert &\leq \mu_{(t)}+\left\lvert \lvert \mathbf{a}^{H}_{i}\mathbf{z}^{(t)}\rvert-\lvert \mathbf{a}^{H}_{i}\mathbf{x} \rvert\right\rvert \nonumber\\
	& \leq \mu_{(0)} + \left\lvert  e^{-j\theta_{(t)}}\mathbf{a}^{H}_{i}\mathbf{z}^{(t)}- \mathbf{a}^{H}_{i}\mathbf{x}\right\rvert \nonumber\\
	&=\mu_{(0)} + \lvert \mathbf{a}^{H}_{i}\mathbf{h}^{(t)}\rvert,
	\label{eq:inequality5}
	\end{align} 
	in which the second line comes after the triangular inequality. Then, putting together \eqref{eq:normBt1} and \eqref{eq:inequality5} one can conclude that
	\begin{align}
		\left\lvert\frac{e^{-j\theta_{(t)}}\mathbf{a}_{i}^{H}\mathbf{z}^{(t)}}{\sqrt{\lvert\mathbf{a}_{i}^{H}\mathbf{z}^{(t)}\rvert^{2}+\mu_{(t)}^{2}}}-\frac{\mathbf{a}_{i}^{H}\mathbf{x} }{\lvert \mathbf{a}_{i}^{H}\mathbf{x} \rvert}\right\rvert \lvert\mathbf{a}_{i}^{H}\mathbf{x}\rvert \leq \mu_{(0)} + 2\lvert \mathbf{a}^{H}_{i}\mathbf{h}^{(t)}\rvert.
		\label{eq:inequality6}
	\end{align}
	
	Combining \eqref{eq:normBt} and \eqref{eq:inequality6} it can be obtained that
	\begin{align}
	\frac{1}{m}\left\lVert\mathbf{b}^{(t)}\right\rVert_{2}^{2} \leq  \frac{1}{m} \sum_{i=1}^{m} \left( \mu_{(0)} + 2\lvert \mathbf{a}^{H}_{i}\mathbf{h}^{(t)}\rvert\right)^{2} = \frac{4}{m}\sum_{i=1}^{m} \lvert \mathbf{a}^{H}_{i}\mathbf{h}^{(t)}\rvert^{2} + \mu_{(0)}c,
	\label{eq:inequality7}
	\end{align}
	where $c = \mu_{(0)}\left(\frac{4}{m}\sum_{i=1}^{m} \lvert \mathbf{a}^{H}_{i}\mathbf{h}^{(t)}\rvert + \mu_{(0)} \right)$. Applying Lemma 7.8 in \cite{candWir}, we have that if $m\geq c_{0}\epsilon_{4}^{-2}n$, then with probability $1-2e^{-\epsilon_{4}^{2} m/2}$
	\begin{equation}
		(1-\epsilon_{4})\lVert \mathbf{h}^{(t)} \rVert_{2}^{2}\leq \frac{1}{m}\sum_{i=1}^{m} \lvert\mathbf{a}_{i}^{H}\mathbf{h}^{(t)}\rvert^{2} \leq (1+\epsilon_{4})\lVert \mathbf{h}^{(t)} \rVert_{2}^{2},
	\label{eq:approxH}
	\end{equation}
	holds for all vectors $\mathbf{h}^{(t)}$ and for any $\epsilon_{4} \in (0,1)$. Then, by combining \eqref{eq:inequality7} and \eqref{eq:approxH} it can be obtained that
	\begin{align}
		\frac{1}{m}\left\lVert\mathbf{b}^{(t)}\right\rVert_{2}^{2} \leq 4(1+\epsilon_{4}) \lVert \mathbf{h}^{(t)} \rVert_{2}^{2} + \mu_{(0)}c
	\label{eq:final1}
	\end{align}
	with probability at least $1-2e^{-\epsilon_{4}^{2} m/2}$.
		
	Notice that inequality in \eqref{eq:final1} is satisfied for all initial $\mu_{(0)}\in \mathbb{R}_{++}$. Then, by Theorem 1.1 in \cite{apostol1974mathematical}, one can conclude that 
	\begin{align}
	\label{eq:thirdTerm}
	\frac{1}{m}\left\lVert\mathbf{b}^{(t)}\right\rVert_{2}^{2} \leq& 4(1+\epsilon_{4})\lVert\mathbf{h}_{\Theta^{(t)}}^{(t)}\rVert^{2}_{2}\\ \nonumber
	\left\lVert\frac{1}{\sqrt{m}}\mathbf{b}^{(t)}\right\rVert_{2} \leq & 2(1+\epsilon_{5})\lVert\mathbf{h}_{\Theta^{(t)}}^{(t)}\rVert_{2},
	\end{align}
	for any $\epsilon_{5}>0$ with probability at least $1-2e^{-\epsilon_{5}^{2} m/2}$.
	
	Therefore, putting together the bounds in \eqref{eq:bound1}, \eqref{eq:secondTerm}, \eqref{eq:hterm} and \eqref{eq:thirdTerm} into \eqref{eq:inequality2}, one can write
	\begin{align}
		\nonumber \frac{1}{2}\lVert \mathbf{h}^{(t+1)} \rVert_{2} \leq & \max\{1-2\tau(1-\zeta_{1}-\epsilon_{1}),2\tau(1+\epsilon_{0})-1 \}\lVert \mathbf{h}_{\Theta^{(t+1)}}^{(t)}\rVert_{2} \\\nonumber
		+& 2\tau(\zeta_{2}+\epsilon_{2})\lVert \mathbf{h}_{\Theta^{(t)}\setminus\Theta^{(t+1)}}^{(t)}\rVert_{2} + 4\tau(1+\epsilon_{3})(1+\epsilon_{5})\lVert\mathbf{h}_{\Theta^{(t)}}^{(t)}\rVert_{2} \\ \nonumber
		\leq & \sqrt{2}\max\{ \vartheta , 2\tau(\zeta_{2}+\epsilon_{2}) \}\lVert \mathbf{h}^{(t)} \rVert_{2} + 4\tau(1+\epsilon_{3})(1+\epsilon_{5})\lVert\mathbf{h}^{(t)}\rVert_{2} \\
		\lVert \mathbf{h}^{(t+1)} \rVert_{2} \leq & \rho\lVert\mathbf{h}^{(t)}\rVert_{2},
		\end{align}
	in which the second inequality results from $\lVert \mathbf{h}_{\Theta^{(t+1)}}^{(t)} \rVert_{2} + \lVert \mathbf{h}_{\Theta^{(t)}\setminus\Theta^{(t+1)}}^{(t)}\rVert_{2}\leq \sqrt{2}\lVert \mathbf{h}^{(t)}\rVert_{2}$, with $\vartheta=\max\{1-2\tau(1-\zeta_{1}-\epsilon_{1}),2\tau(1+\epsilon_{0})-1 \}$. From the last inequality it can be obtained that 
	\begin{equation}
		\rho = 2\left( \sqrt{2}\max\{ \vartheta , 2\tau(\zeta_{2}+\epsilon_{2}) \} + 4\tau(1+\epsilon_{3})(1+\epsilon_{5})\right).
	\label{eq:rhoequ}
	\end{equation}
	Then, to ensure linear convergence, from \eqref{eq:rhoequ} it suffices to choose a step $\tau>0$ such that $\rho<1$ in \eqref{eq:rhoequ}. Letting $\eta=1-\rho\in (0,1)$, which justifies the linear convergence result in \eqref{eq:linearConvergence} with probability exceeding $1-2 e^{-c_{1}m}$ for some $c_{1}\geq 0$.
\end{proof}

\section*{Appendix C: Proof of Theorem \ref{theo:assumption}}
\label{app:proofTheo2}
\begin{proof}
	1) Suppose that $S_{\mu}(\mathbf{w})$ in Eq.(24) is unbounded, then there exists a sequence $\{\mathbf{x}_{\ell}\}\subseteq S_{\mu}(\mathbf{w})$ such that $\lVert \mathbf{x}_{\ell} \rVert_{2}\rightarrow \infty$. From the definition of the level set $S_{\mu}(\mathbf{w})$, it can be obtained that
	\begin{equation}
	g(\mathbf{x}_{\ell},\mu)\leq g(\mathbf{w},\mu)< \infty, \forall \ell\in \mathbb{N}.
	\end{equation}
	However, we assume that $\text{span}(\mathbf{a}_{1},\cdots,\mathbf{a}_{m})=\mathbb{C}^{n}$, then the fact that $\lVert \mathbf{x}_{\ell} \rVert_{2}\rightarrow \infty$ implies that the sequence $g(\mathbf{x}_{\ell},\mu)\rightarrow \infty$ according to the definition of the function $g$ in \eqref{eq:optimizationProblem}. Then  $g(\mathbf{x}_{\ell},\mu)\rightarrow \infty$ is a contradiction, because $g(\mathbf{x}_{\ell},\mu)<\infty$, $\forall \ell\in \mathbb{N}$. Thus, $S_{\mu}(\mathbf{w})$ is a bounded set.
	
	\vspace{1em}
	To prove the second part of Assumption 1, we proceed to show that for each function $h_{k,\mu}(\mathbf{x})=\left(\varphi_{\mu}(\lvert\mathbf{a}_{k}^{H}\mathbf{x}\rvert)-q_{k}\right)^{2}$ its Wirtinger derivative is Lipschitz. Thus, since $g(\mathbf{x},\mu)$ is the sum of the functions $h_{k,\mu}(\mathbf{x})$, then the Writinger derivative of $g(\mathbf{x},\mu)$ is Lipschitz as it is proven in Chapter 12 in \cite{eriksson2013applied}. 
	
	Notice that, the Wirtinger derivative of $h_{k,\mu}$ at point $\mathbf{w}\in \mathbb{C}^{n}$ is given by
	\begin{align}
	\partial h_{k,\mu}(\mathbf{w}) &= 2\left(\varphi_{\mu}\left(\lvert\mathbf{a}_{k}^{H}\mathbf{w}\rvert\right)-q_{k}\right)\frac{\mathbf{a}^{H}_{k}\mathbf{w}}{\varphi_{\mu}(\lvert \mathbf{a}_{k}^{H}\mathbf{w}\rvert)}\mathbf{a}_{k}\nonumber\\
	&=2\left((\mathbf{a}_{k}^{H}\mathbf{w})\mathbf{a}_{k}-q_{k}\frac{\mathbf{a}_{k}^{H}\mathbf{w}}{\varphi_{\mu}(\lvert \mathbf{a}_{k}^{H}\mathbf{w}\rvert)}\mathbf{a}_{k}\right).
	\label{eq:gradientH}
	\end{align}
	By definition of $d_{r}(\cdot,\cdot)$ in Eq.(3), it can be obtained that
	\begin{equation}
	\small{d_{r}(\partial h_{k,\mu}(\mathbf{w}_{1}),\partial h_{k,\mu}(\mathbf{w}_{2})) \leq \lVert e^{-j\theta}\partial h_{k,\mu}(\mathbf{w}_{1})-\partial h_{k,\mu}(\mathbf{w}_{2}) \rVert_{2},}
	\label{eq:ineqLips}			
	\end{equation}
	for any $\mathbf{w}_{1},\mathbf{w}_{2}\in S_{\mu}(\mathbf{w})$ and $\theta\in [0,2\pi)$. Then, combining \eqref{eq:gradientH} and \eqref{eq:ineqLips}, one can write that
	\small{\begin{align}
		d_{r}(\partial h_{k,\mu}(\mathbf{w}_{1}),\partial h_{k,\mu}(\mathbf{w}_{2})) &\leq 2\lVert \mathbf{a}_{k}\rVert_{2}\left\lvert e^{-j\theta}(\mathbf{a}_{k}^{H}\mathbf{w}_{1}) - \mathbf{a}_{k}^{H}\mathbf{w}_{2}  \right \rvert\nonumber\\
		&+2q_{k}\lVert \mathbf{a}_{k}\rVert_{2} \left\lvert \frac{e^{-j\theta}(\mathbf{a}^{H}_{k}\mathbf{w}_{1})}{\varphi_{\mu}(\lvert \mathbf{a}_{k}^{H}\mathbf{w}_{1}\rvert)} - \frac{\mathbf{a}^{H}_{k}\mathbf{w}_{2}}{\varphi_{\mu}(\lvert \mathbf{a}_{k}^{H}\mathbf{w}_{2}\rvert)}\right\rvert \nonumber \\
		&\leq 2\lVert \mathbf{a}_{k}\rVert_{2}^{2}\lVert e^{-j\theta}\mathbf{w}_{1}-\mathbf{w}_{2}\rVert_{2} \nonumber\\
		&+\frac{2q_{k}\lVert \mathbf{a}_{k}\rVert_{2}}{\mu^{2}} \left\lvert e^{-j\theta}(\mathbf{a}^{H}_{k}\mathbf{w}_{1})\varphi_{\mu}(\lvert \mathbf{a}_{k}^{H}\mathbf{w}_{2}\rvert) -  (\mathbf{a}^{H}_{k}\mathbf{w}_{2})\varphi_{\mu}(\lvert \mathbf{a}_{k}^{H}\mathbf{w}_{1}\rvert) \right\rvert,
		\label{eq:gradientLips1}
		\end{align}}\normalsize
	where the first inequality is obtained using the triangular inequality and the second comes from the fact that $\varphi_{\mu}(t)\geq \mu>0$ for all $t\in \mathbb{R}$, and using the Cauchy-Schwarz inequality. Then, from \eqref{eq:gradientLips1} it can be obtained that
	\small{\begin{align}
	\left\lvert e^{-j\theta}(\mathbf{a}^{H}_{k}\mathbf{w}_{1})\varphi_{\mu}(\lvert \mathbf{a}_{k}^{H}\mathbf{w}_{2}\rvert) -  (\mathbf{a}^{H}_{k}\mathbf{w}_{2})\varphi_{\mu}(\lvert \mathbf{a}_{k}^{H}\mathbf{w}_{1}\rvert) \right\rvert &\leq\left \lvert \varphi_{\mu}(\lvert  \mathbf{a}_{k}^{H}\mathbf{w}_{2}\rvert) \left[e^{-j\theta}(\mathbf{a}^{H}_{k}\mathbf{w}_{1})-\mathbf{a}^{H}_{k}\mathbf{w}_{2}\right] \right \rvert \nonumber\\
	&+\left\lvert (\mathbf{a}_{k}^{H}\mathbf{w}_{2})\left[\varphi_{\mu}(\lvert \mathbf{a}_{k}^{H}\mathbf{w}_{1}\rvert)-\varphi_{\mu}(\lvert \mathbf{a}_{k}^{H}\mathbf{w}_{2}\rvert)\right] \right\rvert\nonumber\\
	&\leq M_{\varphi_{\mu}}\lVert \mathbf{a}_{k}\rVert_{2}\lVert e^{-j\theta}\mathbf{w}_{1}-\mathbf{w}_{2}\rVert_{2} \nonumber\\
	&+ M_{S_{\mu}}\lVert \mathbf{a}_{k}\rVert_{2} \left\lvert \varphi_{\mu}(\lvert \mathbf{a}_{k}^{H}\mathbf{w}_{1}\rvert)-\varphi_{\mu}(\lvert \mathbf{a}_{k}^{H}\mathbf{w}_{2}\rvert) \right\rvert,
	\label{eq:gradientLips2}
	\end{align}}\normalsize
	where the second inequality is obtained using the triangular inequality and the following two reasons. First, $\varphi_{\mu}(\lvert \mathbf{a}_{k}^{H}\mathbf{z}\rvert)$ is a bounded function in $S_{\mu}(\mathbf{w})$ for any $\mathbf{z}\in S_{\mu}(\mathbf{w})$, since $S_{\mu}(\mathbf{w})$ is a bounded set as was established in the previous item, $i.e.$ $\varphi_{\mu}(\lvert \mathbf{a}_{k}^{H}\mathbf{z}\rvert)\leq M_{\varphi_{\mu}}$ for some constant $M_{\varphi_{\mu}}\in \mathbb{R}_{+}$. Second, $S_{\mu}(\mathbf{w})$ is a bounded set, then $\lVert\mathbf{z}\rVert_{2}\leq M_{S_{\mu}}, \forall \mathbf{z}\in S_{\mu}(\mathbf{w})$ for some constant $M_{S_{\mu}}\in \mathbb{R}_{+}$. Hence, considering that $\varphi_{\mu}(\cdot)$ is a Lipschitz function with constant $L_{\varphi_{\mu}}=1$, then from \eqref{eq:gradientLips2} we have that
	\begin{align}
	\left\lvert \varphi_{\mu}(\lvert \mathbf{a}_{k}^{H}\mathbf{w}_{1}\rvert)-\varphi_{\mu}(\lvert \mathbf{a}_{k}^{H}\mathbf{w}_{2}\rvert) \right\rvert &\leq \left\lvert \lvert \mathbf{a}_{k}^{H}\mathbf{w}_{1}\rvert-\lvert \mathbf{a}_{k}^{H}\mathbf{w}_{2}\rvert \right\rvert \nonumber\\
	&\leq \left\lvert e^{-j\theta}(\mathbf{a}_{k}^{H}\mathbf{w}_{1})-\mathbf{a}_{k}^{H}\mathbf{w}_{2} \right\rvert \nonumber\\
	&\leq \lVert \mathbf{a}_{k}\rVert_{2}\lVert e^{-j\theta}\mathbf{w}_{1}-\mathbf{w}_{2}\rVert_{2},
	\label{eq:gradientLips3}
	\end{align}
	where the second and third lines come from the triangular and Cauchy-Schwarz inequality, respectively, and it is valid for all $\theta\in [0,2\pi)$. Therefore, combining \eqref{eq:gradientLips1}, \eqref{eq:gradientLips2} and \eqref{eq:gradientLips3}, we have that
	\begin{align}
	d_{r}(\partial h_{k,\mu}(\mathbf{w}_{1}),\partial h_{k,\mu}(\mathbf{w}_{2})) \leq \tilde{L}_{h_{k,\mu}}\left\lVert  e^{-j\theta}\mathbf{w}_{1}- \mathbf{w}_{2} \right\rVert_{2},
	\label{eq:gradientLips4}
	\end{align}
	with $\tilde{L}_{h_{k,\mu}}=2\lVert \mathbf{a}_{k}\rVert_{2}^{2} + \frac{2q_{k}M_{\varphi_{\mu}}\lVert \mathbf{a}_{k}\rVert^{2}_{2}}{\mu^{2}} + \frac{2q_{k}M_{S_{\mu}}\lVert \mathbf{a}_{k}\rVert^{3}_{2}}{\mu^{2}}$. Notice that, for the $i.i.d.$ Gaussian vectors $\mathbf{a}_{k}$, $\lVert \mathbf{a}_{k}\rVert^{2}_{2}\leq 2.3n$ holds with probability at least $1-me^{-n/2}$ \cite{wang2016solving}. Then, we have that $\tilde{L}_{h_{k,\mu}}\leq 4.6n+\frac{4.6q_{k}n}{\mu^{2}}+\frac{13n^{3/2}M_{S_{\mu}}}{\mu^{2}}=L_{h_{k,\mu}}$ with probability exceeding $1-me^{-n/2}$. Further, taking the value of $\theta$ that minimizes the term $\lVert e^{-j\theta}\mathbf{w}_{1}-\mathbf{w}_{2} \rVert_{2}$, from \eqref{eq:gradientLips4}, it can be concluded that
	\begin{align}
	d_{r}(\partial h_{k,\mu}(\mathbf{w}_{1}),\partial h_{k,\mu}(\mathbf{w}_{2})) \leq L_{h_{k,\mu}}d_{r}(\mathbf{w}_{1}, \mathbf{w}_{2}).
	\label{eq:gradientLips5}
	\end{align}
	Then, since $g(\mathbf{x},\mu)$ is the sum of the functions $h_{k,\mu}(\mathbf{x})$, then the Writinger derivative of $g(\mathbf{x},\mu)$ is Lipschitz as it is proven in Chapter 12 in \cite{eriksson2013applied}. Thus, from \eqref{eq:gradientLips5} the result holds.
\end{proof}

\section*{Appendix D: Proof of Theorem 3}
Before to prove Theorem \ref{eq:convergence}, we need to introduce first the contraction mapping definition and the Hahn Banach Fixed Point theorem as follows.

\begin{definition}{\textit{Contraction mapping:}}
	Let $f: (\mathbb{C}^{n},d_{r}(\cdot,\cdot))\rightarrow \mathbb{R}$ be a function. Then,  $f(\mathbf{x})$ is a contraction mapping if there is some nonnegative $\beta \in [0,1)$ such that 
	\begin{equation}
	d_{r}(f(\mathbf{x}),f(\mathbf{y})) \leq \beta d_{r}(\mathbf{x},\mathbf{y}), \forall \mathbf{x},\mathbf{y} \in \mathbb{C}^{n}.
	\end{equation}
	\label{def:contraction}
\end{definition}

\begin{theorem}{\textit{Hahn Banach Fixed Point:}}
	\label{theo:banach}
	Let  $f: (\mathbb{C}^{n},d_{r}(\cdot,\cdot))\rightarrow \mathbb{R}$ be a contraction mapping. Then $f(\mathbf{x})$ admits a unique fixed-point $\mathbf{x}^{*}\in \mathbb{C}^{n}$. ($i.e.$ $f(\mathbf{x}^{*})=\mathbf{x}^{*}$). (The proof of Theorem \ref{theo:banach} can be found in \cite{kreyszig1989introductory})
\end{theorem}

\begin{proof}
	Define $\mathcal{M}=\{t \mbox{ }| \mbox{ } \mu_{(t+1)}=\gamma_{1}\mu_{(t)} \}$ and suppose that $\mathcal{M}$ is a finite set. Then, according to Line 5 in Algorithm \ref{alg:smothing} there exists an integer $\overline{t}$ such that $\lVert \partial g(\mathbf{z}^{(t)},\mu_{(t-1)}) \rVert_{2} \geq \gamma \mu_{(t-1)}$ for all $t>\overline{t}$. Taking $\overline{\mu}=\mu_{(\overline{t})}$, the optimization problem in \eqref{eq:optimizationProblem} solved by Algorithm \ref{alg:smothing}, reduces to solve 
	\begin{equation}
	\min_{\lVert \mathbf{x} \rVert_{0}=k} \frac{1}{m}\sum_{i=1}^{m} \left(\varphi_{\overline{\mu}}(\lvert\mathbf{a}_{i}^{H}\mathbf{x}\rvert)-q_{i}\right)^{2}.
	\label{eq:secondOpti}
	\end{equation}
	Notice that, Theorem \ref{theo:contrac} provides that the sequence $\{\mathbf{z}^{(t)}\}_{t\geq 1}$, generated by Algorithm \ref{alg:smothing} in Line 4 produces a monotonically decreasing sequence $\{g(\mathbf{z}^{(t)},\overline{\mu})\}_{t\geq 1}$, for the fixed $\overline{\mu}$. Further, from \ref{eq:linearConvergence} it can be obtained that 
	\begin{equation}
	d_{r}(\mathbf{z}^{(t+1)},\mathbf{x}) \leq (1-\eta)d_{r}(\mathbf{z}^{(t)},\mathbf{x}), \forall t\in \mathbb{N},
	\label{eq:localContra}
	\end{equation}
	where $\eta \in (0,1)$. Then, from \eqref{eq:localContra} it can be concluded that the thresholding step in Algorithm \ref{alg:smothing} is contractive according to Definition \ref{def:contraction}. Then, from Theorem \ref{theo:banach} it can be obtained that there exists a fixed point, which means that $\mathbf{z}^{(t_{1}+1)} = \mathbf{z}^{(t_{1})}$, for some $t_{1}\in \mathbb{N}$. Then, considering this previous condition and the thresholded step of the reduced optimization problem in \eqref{eq:secondOpti}, it can be obtained that
	\begin{equation}
	\mathbf{z}^{(t_{1}+1)} = \mathcal{H}_{k}(\mathbf{z}^{(t_{1})} - \tau\partial g(\mathbf{z}^{(t_{1})},\overline{\mu}) )  = \mathbf{z}^{(t_{1})}.
	\label{eq:fixedPoint}
	\end{equation}
	Thus, from \eqref{eq:fixedPoint} it can be concluded such that 
	\begin{equation}
	\liminf_{t\rightarrow \infty}\lVert \partial g(\mathbf{z}^{(t)},\mu_{(t-1)}) \rVert_{2}=0,
	\end{equation}
	which contradicts the fact that $\lVert \partial g(\mathbf{z}^{(t)},\mu_{(t-1)}) \rVert_{2} \geq \gamma \mu_{(t-1)}$. This shows that $\mathcal{M}$ must be infinite and $\lim_{t\rightarrow \infty} \mu_{(t)}=0$. Thus, since $\mathcal{M}$ is infinite one can assume $\mathcal{M}=\{t_{0},t_{1},\cdots \}$ with $t_{0}<t_{1}<\cdots$. Then, it can be expressed that
	\begin{equation}
	\liminf_{t\rightarrow \infty}\lVert \partial g(\mathbf{z}^{(t)},\mu_{(t-1)}) \rVert_{2}\leq \gamma \lim_{t\rightarrow \infty } \mu_{(t)}=0,
	\end{equation}
	which is the desired result.
\end{proof}

\section*{Appendix E: Proof of Theorem 4}
\label{app:smoothpro}
\begin{proof} 
	Let $\mathbf{h} = e^{-j\theta(z)}\mathbf{z}-\mathbf{x}$ with $\theta(z) = \argmin_{\theta\in[0,2\pi) }\lVert e^{-j\theta}\mathbf{z}-\mathbf{x} \rVert_{2}$, for a given $\mathbf{z}\in \mathbb{C}^{n}$. Then, by definition of $d_{r}(\cdot,\cdot)$ in \eqref{eq:distance} we have that $d_{r}(\mathbf{z},\mathbf{x})=\lVert\mathbf{h} \rVert_{2}$. From \eqref{eq:gradient} it can be obtained that
	\begin{equation}
	\lVert \partial g(\mathbf{z},\mu)\rVert_{2} = \left \lVert \frac{2}{m}\sum_{i=1}^{m}\left(\mathbf{a}_{i}^{H}\mathbf{z} - q_{i}\frac{\mathbf{a}_{i}^{H}\mathbf{z}}{\sqrt{\lvert\mathbf{a}_{i}^{H}\mathbf{z}\rvert^{2}+\mu^{2}}}\right)\mathbf{a}_{i} \right\rVert_{2},
	\label{eq:gradient1}
	\end{equation}
	Notice that \eqref{eq:gradient1} can be rewritten as
	\begin{equation}
	\lVert \partial g(\mathbf{z},\mu)\rVert_{2} = \left \lVert \frac{2}{m}\sum_{i=1}^{m}\left(e^{-j\theta(z)}\mathbf{a}_{i}^{H}\mathbf{z} - q_{i}\frac{e^{-j\theta(z)}\mathbf{a}_{i}^{H}\mathbf{z}}{\sqrt{\lvert\mathbf{a}_{i}^{H}\mathbf{z}\rvert^{2}+\mu^{2}}}\right)\mathbf{a}_{i} \right\rVert_{2}.
	\end{equation}
	Then, note that
	\begin{align}
		\lVert \partial g(\mathbf{z},\mu)\rVert_{2} &= \left\lVert \frac{2}{m}\sum_{i=1}^{m} \left(\mathbf{a}_{i}^{H}\mathbf{h} - q_{i}\left[\frac{e^{-j\theta(z)}\mathbf{a}_{i}^{H}\mathbf{z}}{\sqrt{\lvert\mathbf{a}_{k}^{H}\mathbf{z}\rvert^{2}+\mu^{2}}} - \frac{\mathbf{a}_{i}^{H}\mathbf{x}}{\lvert\mathbf{a}_{i}^{H}\mathbf{x}\rvert} \right]\right)\mathbf{a}_{i} \right\rVert_{2} \nonumber \\
		&\leq \left\lVert \frac{2}{m}\sum_{i=1}^{m} \mathbf{a}_{i}\mathbf{a}_{i}^{H}\mathbf{h} \right\rVert_{2} + \frac{2}{m}\sum_{i=1}^{m} q_{i} \left\lVert \mathbf{a}_{i} \right\rVert_{2}\upsilon_{i},
	\end{align}
	where $\upsilon_{i}= \left\lvert \frac{e^{-j\theta(z)}\mathbf{a}_{i}^{H}\mathbf{z}}{\sqrt{\lvert\mathbf{a}_{i}^{H}\mathbf{z}\rvert^{2}+\mu^{2}}} - \frac{\mathbf{a}_{i}^{H}\mathbf{x}}{\lvert\mathbf{a}_{i}^{H}\mathbf{x}\rvert} \right\rvert$. Then, from the above inequality it can be obtained that
	\begin{align}
	\lVert \partial g(\mathbf{z},\mu)\rVert_{2}&\leq 2 \left\lVert \frac{1}{m}\sum_{i=1}^{m} \mathbf{a}_{i}\mathbf{a}_{i}^{H} \right\rVert_{2\rightarrow 2}\lVert \mathbf{h} \rVert_{2} + \frac{2\sqrt{2.3n}}{m}\sum_{i=1}^{m} q_{i}\upsilon_{i}
	\label{eq:smooth1}
	\end{align}
	where $\lVert \cdot \rVert_{2\rightarrow 2}$ represents the spectral norm and for the $i.i.d.$ Gaussian vectors $\mathbf{a}_{i}$, $\max_{i\in\{1,\cdots,m \}}\lVert \mathbf{a}_{i}\rVert_{2}\leq \sqrt{2.3n}$ holds with probability at least $1-me^{-n/2}$ \cite{wang2016solving}. Notice that by corollary 5.35 in \cite{vershynin2010introduction} it can be obtained that
	\begin{equation}
	\left\lVert \frac{1}{m}\sum_{i=1}^{m} \mathbf{a}_{i}\mathbf{a}_{i}^{H} \right\rVert_{2\rightarrow 2} \leq 1+\epsilon_{0},
	\label{eq:approximation11}
	\end{equation}
	with probability at least $1-2e^{-me_{0}^{2}/2}$ when $m\geq C(\epsilon_{0})n$ for some constant $C(\epsilon_{0})$ depending on $\epsilon_{0}>0$. Thus, from \eqref{eq:smooth1} and \eqref{eq:approximation11} it can be concluded that
	\begin{equation}
	\lVert \partial g(\mathbf{z},\mu)\rVert_{2} \leq 2 (1+\epsilon_{0})\lVert \mathbf{h} \rVert_{2} +  \frac{2\sqrt{2.3n}}{m}\sum_{i=1}^{m} q_{i}\upsilon_{i}.
		\label{eq:smooth2}
	\end{equation}
	On the other hand, notice that from \eqref{eq:smooth2} it can be obtained that
	\begin{align}
	\sum_{i=1}^{m} q_{i}\upsilon_{i} &\leq \sum_{i=1}^{m}q_{i} \left\lvert \frac{e^{-j\theta(z)}\mathbf{a}_{i}^{H}\mathbf{z}}{\sqrt{\lvert\mathbf{a}_{i}^{H}\mathbf{z}\rvert^{2}+\mu^{2}}} - \frac{e^{-j\theta(z)}\mathbf{a}_{i}^{H}\mathbf{z}}{\lvert\mathbf{a}_{i}^{H}\mathbf{x}\rvert}  \right\rvert  + \sum_{i=1}^{m}q_{i} \left\lvert \frac{e^{-j\theta(z)}\mathbf{a}_{i}^{H}\mathbf{z}}{\lvert\mathbf{a}_{i}^{H}\mathbf{x}\rvert} - \frac{\mathbf{a}_{i}^{H}\mathbf{x}}{\lvert\mathbf{a}_{i}^{H}\mathbf{x} \rvert} \right\rvert \nonumber \\
	&\leq \sum_{i=1}^{m}\left\lvert\sqrt{\lvert \mathbf{a}_{i}^{H}\mathbf{z} \rvert^{2}+\mu^{2}} -\lvert \mathbf{a}^{H}_{i}\mathbf{x}\rvert\right\rvert +\sum_{i=1}^{m}\lvert \mathbf{a}_{i}^{H}\mathbf{h}\rvert,
	\label{eq:inequality3}
	\end{align}
	in which the second inequality comes from the fact that
	\begin{align}
	q_{i}\left\lvert \frac{e^{-j\theta(z)}\mathbf{a}_{i}^{H}\mathbf{z}}{\sqrt{\lvert\mathbf{a}_{i}^{H}\mathbf{z}\rvert^{2}+\mu^{2}}} - \frac{e^{-j\theta(z)}\mathbf{a}_{i}^{H}\mathbf{z}}{\lvert\mathbf{a}_{i}^{H}\mathbf{x}\rvert} \right\rvert &\leq \frac{q_{i}\lvert \mathbf{a}^{H}_{i}\mathbf{z}\rvert}{\lvert\mathbf{a}_{i}^{H}\mathbf{x}\rvert\sqrt{\lvert\mathbf{a}_{i}^{H}\mathbf{z}\rvert^{2}+\mu^{2}}} \left\lvert\sqrt{\lvert \mathbf{a}_{i}^{H}\mathbf{z} \rvert^{2}+\mu^{2}} -\lvert \mathbf{a}^{H}_{i}\mathbf{x}\rvert\right\rvert \nonumber\\
	&\leq \left\lvert\sqrt{\lvert \mathbf{a}_{i}^{H}\mathbf{z} \rvert^{2}+\mu^{2}} -\lvert \mathbf{a}^{H}_{i}\mathbf{x}\rvert\right\rvert.
	\label{eq:inequality44}
	\end{align}
	Then, from \eqref{eq:inequality44} it can be obtained that
	\begin{align}
	\left\lvert\sqrt{\lvert \mathbf{a}_{i}^{H}\mathbf{z} \rvert^{2}+\mu^{2}} -\lvert \mathbf{a}^{H}_{i}\mathbf{x}\rvert\right\rvert &\leq \mu +\lvert \lvert \mathbf{a}^{H}_{i}\mathbf{z}\rvert-\lvert \mathbf{a}^{H}_{i}\mathbf{x} \rvert\rvert \nonumber\\
	& \leq \mu + \lvert  e^{-j\theta(z)}\mathbf{a}^{H}_{i}\mathbf{z}- \mathbf{a}^{H}_{i}\mathbf{x}\rvert \nonumber\\
	&=\mu + \lvert \mathbf{a}^{H}_{i}\mathbf{h}\rvert,
	\label{eq:inequality55}
	\end{align} 
	in which the second line comes after the triangular inequality. Thus, combining \eqref{eq:smooth2}, \eqref{eq:inequality3} and \eqref{eq:inequality55} it can be obtained that
	\begin{equation}
	\lVert \partial g(\mathbf{z},\mu)\rVert_{2} \leq 2 (1+\epsilon_{0})\lVert \mathbf{h} \rVert_{2} +  \frac{4\sqrt{2.3n}}{m}\sum_{i=1}^{m} \lvert \mathbf{a}^{H}_{i}\mathbf{h}\rvert + 2\sqrt{2.3n}\mu.
	\label{eq:smooth3}
	\end{equation}
	Note that the inequality in \eqref{eq:smooth3} is satisfied for all $\mu\in \mathbb{R}_{++}$. Then by Theorem 1.1 in \cite{apostol1974mathematical}, one can conclude that
	\begin{equation}
	\lVert \partial g(\mathbf{z},\mu)\rVert_{2} \leq \beta \lVert \mathbf{h} \rVert_{2} +  \frac{\rho}{m}\sum_{i=1}^{m} \lvert \mathbf{a}^{H}_{i}\mathbf{h}\rvert,
	\label{eq:inequality8}
	\end{equation}
	where $\beta = 2 (1+\epsilon_{0})$, $\rho = 4\sqrt{2.3n}$, with probability $1-me^{-n/2}$. Thus, given the fact that $d_{r}(\mathbf{z},\mathbf{x}) = \lVert \mathbf{h} \rVert_{2}$, then from \eqref{eq:inequality8} we can conclude that
	\begin{align}
	\lVert \partial g(\mathbf{z},\mu)\rVert_{2} \leq \beta d_{r}(\mathbf{z},\mathbf{x}) +  \frac{\rho}{m}\sum_{i=1}^{m} \lvert \mathbf{a}^{H}_{i}\mathbf{h}\rvert,
	\label{eq:smoothfinal}
	\end{align}
	with probability at least $1-me^{-n/2}$. Finally, from \eqref{eq:smoothfinal} the result holds.
\end{proof}

\bibliographystyle{spphys}       
\bibliography{sample}

\begin{thebibliography}{10}
\providecommand{\url}[1]{{#1}}
\providecommand{\urlprefix}{URL }
\expandafter\ifx\csname urlstyle\endcsname\relax
  \providecommand{\doi}[1]{DOI \discretionary{}{}{}#1}\else
  \providecommand{\doi}{DOI \discretionary{}{}{}\begingroup
  \urlstyle{rm}\Url}\fi

\bibitem{xu2015overcoming}
Y.~Xu, Z.~Ren, K.K. Wong, K.~Tsia, Overcoming the limitation of phase retrieval
  using gerchberg--saxton-like algorithm in optical fiber time-stretch systems,
  Optics letters \textbf{40}(15), 3595 (2015)

\bibitem{fienup1987phase}
C.~Fienup, J.~Dainty, Phase retrieval and image reconstruction for astronomy,
  Image Recovery: Theory and Application pp. 231--275 (1987)

\bibitem{mayo2003x}
S.~Mayo, T.~Davis, T.~Gureyev, P.~Miller, D.~Paganin, A.~Pogany, A.~Stevenson,
  S.~Wilkins, X-ray phase-contrast microscopy and microtomography, Optics
  Express \textbf{11}(19), 2289 (2003)

\bibitem{millane1990phase}
R.P. Millane, Phase retrieval in crystallography and optics, JOSA A
  \textbf{7}(3), 394 (1990)

\bibitem{pinilla2018coded1}
S.~Pinilla, H.~Garc{\'\i}a, L.~D{\'\i}az, J.~Poveda, H.~Arguello, Coded
  aperture design for solving the phase retrieval problem in x-ray
  crystallography, Journal of Computational and Applied Mathematics
  \textbf{338}, 111 (2018)

\bibitem{pinilla2018coded}
S.~Pinilla, J.~Poveda, H.~Arguello, Coded diffraction system in x-ray
  crystallography using a boolean phase coded aperture approximation, Optics
  Communications \textbf{410}, 707 (2018)

\bibitem{xcrys}
M.~Smyth, J.~Martin, x ray crystallography, Journal of Clinical Pathology
  \textbf{53}(1), 8 (2000)

\bibitem{chen2015solving}
Y.~Chen, E.~Candes, in \emph{Advances in Neural Information Processing Systems}
  (2015), pp. 739--747

\bibitem{candes2015phase}
E.J. Candes, X.~Li, M.~Soltanolkotabi, Phase retrieval from coded diffraction
  patterns, Applied and Computational Harmonic Analysis \textbf{39}(2), 277
  (2015)

\bibitem{fienup1982phase}
J.R. Fienup, Phase retrieval algorithms: a comparison, Applied optics
  \textbf{21}(15), 2758 (1982)

\bibitem{candes2014solving}
E.J. Cand{\`e}s, X.~Li, Solving quadratic equations via phaselift when there
  are about as many equations as unknowns, Foundations of Computational
  Mathematics \textbf{14}(5), 1017 (2014)

\bibitem{candWir}
E.J. Candes, X.~Li, M.~Soltanolkotabi, Phase retrieval via wirtinger flow:
  Theory and algorithms, IEEE Transactions on Information Theory
  \textbf{61}(4), 1985 (2015)

\bibitem{netrapalli2013phase}
P.~Netrapalli, P.~Jain, S.~Sanghavi, in \emph{Advances in Neural Information
  Processing Systems} (2013), pp. 2796--2804

\bibitem{wang2016solving}
G.~Wang, G.B. Giannakis, Y.C. Eldar, Solving systems of random quadratic
  equations via truncated amplitude flow, arXiv preprint arXiv:1605.08285
  (2016)

\bibitem{cande2}
Y.~Chen, E.~Candes, in \emph{Advances in Neural Information Processing Systems}
  (2015), pp. 739--747

\bibitem{sun2016geometric}
J.~Sun, Q.~Qu, J.~Wright, in \emph{Information Theory (ISIT), 2016 IEEE
  International Symposium on} (IEEE, 2016), pp. 2379--2383

\bibitem{jaganathan2015phase}
K.~Jaganathan, Y.C. Eldar, B.~Hassibi, Phase retrieval: An overview of recent
  developments, arXiv preprint arXiv:1510.07713  (2015)

\bibitem{yuan2017phase}
Z.~Yuan, Q.~Wang, H.~Wang, Phase retrieval via sparse wirtinger flow, arXiv
  preprint arXiv:1704.03286  (2017)

\bibitem{wang2016sparse}
G.~Wang, L.~Zhang, G.B. Giannakis, M.~Ak{\c{c}}akaya, J.~Chen, Sparse phase
  retrieval via truncated amplitude flow, arXiv preprint arXiv:1611.07641
  (2016)

\bibitem{zhang2009smoothing}
C.~Zhang, X.~Chen, Smoothing projected gradient method and its application to
  stochastic linear complementarity problems, SIAM Journal on Optimization
  \textbf{20}(2), 627 (2009)

\bibitem{kreyszig1989introductory}
E.~Kreyszig, \emph{Introductory functional analysis with applications}, vol.~1
  (wiley New York, 1989)

\bibitem{hunger2007introduction}
R.~Hunger, \emph{An introduction to complex differentials and complex
  differentiability} (Munich University of Technology, Inst. for Circuit Theory
  and Signal Processing, 2007)

\bibitem{zhang2016reshaped}
H.~Zhang, Y.~Liang, in \emph{Advances in Neural Information Processing Systems}
  (2016), pp. 2622--2630

\bibitem{wang2017solving}
G.~Wang, G.B. Giannakis, Y.~Saad, J.~Chen, Solving almost all systems of random
  quadratic equations, arXiv preprint arXiv:1705.10407  (2017)

\bibitem{candes2008introduction}
E.J. Cand{\`e}s, M.B. Wakin, An introduction to compressive sampling, IEEE
  signal processing magazine \textbf{25}(2), 21 (2008)

\bibitem{eriksson2013applied}
K.~Eriksson, D.~Estep, C.~Johnson, \emph{Applied mathematics: Body and soul:
  Volume 1: Derivatives and geometry in IR3} (Springer Science \& Business
  Media, 2013)

\bibitem{vershynin2010introduction}
R.~Vershynin, Introduction to the non-asymptotic analysis of random matrices,
  arXiv preprint arXiv:1011.3027  (2010)

\bibitem{blumensath2009iterative}
T.~Blumensath, M.E. Davies, Iterative hard thresholding for compressed sensing,
  Applied and computational harmonic analysis \textbf{27}(3), 265 (2009)

\bibitem{apostol1974mathematical}
T.M. Apostol, Mathematical analysis  (1974)

\end{thebibliography}

\end{document}